\documentclass[preprint]{imsart}

\usepackage[T1]{fontenc}
\usepackage[latin1]{inputenc}
\usepackage[english]{babel}
\usepackage{amsmath,amsthm,amssymb,amsfonts,epic,latexsym,hyperref,mathtools, bbm
}
\usepackage{mathabx}
\usepackage[usenames]{color}
\usepackage{cancel}
\usepackage[left=30mm, right=30mm, bottom=30mm]{geometry}
\usepackage[shortlabels]{enumitem}

\makeatletter
\def\@fnsymbol#1{\ensuremath{
\ifcase#1
\or 
* 
\or 
\mathparagraph
\or
\mathsection
\or 
\ddagger
\or 
\|
\or 
**
\or 
\dagger\dagger
\or 
\ddagger\ddagger 
\else\@ctrerr\fi}}
\makeatother

\allowdisplaybreaks
\newcommand{\am}{{\alpha_{-}}}
\newcommand{\ap}{{\alpha_{+}}}
\newcommand{\R}{\mathbb {R}}
\newcommand{\N}{\mathbb {N}}

\newcommand{\E}{I\!\!E}
\newcommand{\indiq}{\pmb{1}}
\newcommand{\Dd}{D([0,\delta],\mathbb{R})}

\newcommand{\F}{\mathcal{F}}
\newcommand{\norme}[1]{|\!| #1|\!|}

\newtheorem{theorem}{Theorem}[section]

\newtheorem{assumption}[theorem]{Assumption}

\newtheorem{definition}[theorem]{Definition}

\newtheorem{lemma}[theorem]{Lemma}

\newtheorem{proposition}[theorem]{Proposition}

\newtheorem{remark}[theorem]{Remark}

\begin{document}

\begin{frontmatter}

\title{Strong propagation of chaos for systems of interacting particles with nearly stable jumps} % Strong error bounds for the mean field limit of a system of interacting particles with stable jumps 
% for $ \alpha \geq 1 $}

\runtitle{Systems of particles with nearly stable jumps}

\begin{aug}
 % indicate corresponding author with \corref{}

\author{\fnms{Eva} \snm{L\"ocherbach}\thanksref{m1}\ead[label=e4]{eva.locherbach@univ-paris1.fr}},
\author{\fnms{Dasha}
\snm{Loukianova}\thanksref{m2}\ead[label=e5]{dasha.loukianova@univ-evry.fr}}
\author{\fnms{Elisa} 
\snm{Marini}\thanksref{m3}\ead[label=e6]{marini@ceremade.dauphine.fr}}

\address{\thanksmark{m1}Statistique, Analyse et Mod\'elisation Multidisciplinaire, Universit\'e Paris 1 Panth\'eon-Sorbonne, EA 4543 et FR FP2M 2036 CNRS
  \thanksmark{m2}Laboratoire de Math\'ematiques et Mod\'elisation d'\'Evry,  Universit\'e
  d'\'Evry Val d'Essonne, UMR CNRS 8071 
\thanksmark{m3}CEREMADE, UMR CNRS 7534, Universit{\'e} Paris Dauphine-PSL
  }

\runauthor{E. L\"ocherbach and D. Loukianova and E. Marini}

 \end{aug}

\begin{abstract}
We consider a system of $N$ interacting particles, described by SDEs driven by Poisson random measures, where the coefficients depend on the empirical measure of the system. Every particle jumps with a jump rate depending on its position. When this happens, all the other particles of the system receive a small random kick which is distributed according to a heavy tailed random variable belonging to the domain of attraction of an $\alpha-$ stable law and scaled by $N^{-1/\alpha},$ where $0 <  \alpha <2 .$ We call these jumps \textit{collateral jumps}. Moreover, in case $ 0 < \alpha < 1, $ the jumping particle itself undergoes a macroscopic, \textit{main} jump. Such systems appear in the modeling of large neural networks, such as the human brain. 

The particular scaling of the collateral jumps implies that the limit of the empirical measures of the system is random and equals the conditional distribution of one typical particle in the limit system, given the source of common noise. Thus the system exhibits the \textit{conditional propagation of chaos property}. The limit system turns out to be solution of a non-linear SDE, driven by an $ \alpha-$stable process. We prove strong unique existence of the limit system and introduce a suitable coupling to obtain the strong convergence of the finite to the limit system, together with precise error bounds for finite time marginals. 

\end{abstract}

\end{frontmatter}

\section{Introduction}
In the present paper we study the large population limit of the Markov process $X^N = (X^N_t)_{t \geq 0 }, $ $X^N_t = (X^{N, 1 }_t, \ldots , X^{N, N}_t ),$ which takes values in $\R^N$  and has generator $A^N$ given by 
\begin{multline}\label{eq:dynintro}
A^N  \varphi ( x) = 
 \sum_{i=1}^N \partial_{x^i} \varphi (x) b(x^i , \mu^{N,x} )  \\
 + \sum_{i=1}^N f (x^i) \int_{\R} \nu ( du ) \left( \varphi \left( x + \psi ( x^i, \mu^{N,x}) e_i + \sum_{j\neq i } \frac{u}{N^{1/\alpha}} e_j \right) - \varphi ( x) \right) ,
\end{multline}
for any smooth test function $ \varphi : \R^N \to \R .$ In the above formula,  $ x= (x^1, \ldots, x^N) \in \R^N,$ $ \mu^{N,x}  = \frac{1}{N} \sum_{j=1}^N \delta_{x^j } $ is the associated
empirical measure, and  $ e_j $ denotes the $j-$th unit vector in $ \R^N.$ Moreover, $ b (x^i , \mu^{N,x})   $ is a bounded drift function depending both on the position $x^i$ of a fixed particle and on the empirical measure $\mu^{N,x} $ of the total system, $ f : \R \to \R_+$ is a Lipschitz continuous bounded rate function, and $ \nu $ is the law of a heavy-tailed random variable belonging to the domain of attraction of an $\alpha$-stable law (see Definition \ref{def:strongdomain} below); $ \nu $ is supposed to be centered if $\alpha\in (1,2)$. 
Since $f$ is bounded, $ X^N $ is a piecewise deterministic Markov process. In between jumps, any particle $ X^{N, i } $ follows a deterministic flow having drift $ b.$ Each particle jumps at rate $ f (x^i ) $ whenever its current position is $x^i.$ When jumping, it receives an additional kick $ \psi ( x^i, \mu^{N,x} ) $ which is added to its position. Moreover, at the same time, all the other particles in the system receive the same \textit{collateral jump} (\cite{andreis_mckeanvlasov_2018}). This collateral jump is random, it has law $\nu,$ and it is renormalized by $ N^{1/\alpha }, $ where $N$ is the system size. Therefore in our system there is coexistence of main jumps - the jumps of size $ \psi(x^i , \mu^{N,x}) $ - and of random and common small kicks that are received synchronously by all the other particles. Such systems have originally been introduced to model large biological neural nets such as the brain where the collateral jumps correspond to the synaptic weight of a neuron on its postsynaptic partners and the main jumps to the hyperpolarization of a neuron after a spike, see for instance \cite{Brillinger}, \cite{Cessac},  \cite{DGLP}, \cite{Touboul}, see also the recent monograph \cite{GLP}.

In a series of papers \cite{ELL1} and \cite{ELL2}, we have already studied the mean field limit of such systems in a diffusive scaling, that is, when $\alpha=2$ and $ \nu $ is a centered probability measure on $ \R$ having a second moment. In this case, the central limit theorem implies that the large population limit of the system having generator \eqref{eq:dynintro} is given by an infinite exchangeable system evolving according to  
\begin{equation}
\label{eq:dynlimintro}
 X^i_t = X^i_0+  \int_0^t b( X^i_s, \mu_s)    ds +  \int_0^t \psi ( X^i_{s- }, \mu_{s-})  d {Z}^{i}_s + \sigma \int_0^t \sqrt{  \mu_s ( f)  } d W_s ,\; t \geq 0,   i \in \N^*,
\end{equation}
where $\sigma^2 = \int_\R u^2 \nu ( du) $ and where $ W$ is a standard one-dimensional Brownian motion which is common to all particles. 
In \eqref{eq:dynlimintro} above, $ {Z}^i $ is the counting process associated to the jumps of particle $i, $ having intensity $ t \mapsto f (  X^i_{t^- }).$ 

The presence of the Brownian motion $W$ is a source of common noise in the limit system and implies that 
the \textit{conditional propagation of chaos} property holds: in the limit system, particles are conditionally independent, if we condition on $ W.$ In particular, we have shown in \cite{ELL1} that the limit empirical measure $  \mu_s$ is the directing measure of the infinite limit system (see Def. (2.6) in \cite{aldous1983} for the precise definition), which is necessarily given by $ \mu_s = {\mathcal L } ( X^i_s | W_u ,u \le s) ,$ such that the stochastic integral term appearing in \eqref{eq:dynlimintro} is given by 
\[\int_0^t \sqrt{ \E ( f (  X^i_s ) | W_u, u \le s   ) } d W_s.\]

It is a natural question to ask what happens in the situation when $ \nu $ does not belong to the domain of attraction of a normal law but of a stable law of index $ \alpha < 2 .$ The present paper gives an answer to this question. Not surprisingly, the limit Brownian motion will be replaced by a stable process $S^\alpha $ of index $ \alpha $ such that the limit equation is  now given by 
\begin{equation}
\label{eq:dynlimintrovrai}
\bar X^i_t = \bar X^i_0+  \int_0^t b(\bar X^i_s, \bar \mu_s)    ds +  \int_0^t \psi (\bar X^i_{s- }, \bar \mu_{s-})  d\bar{Z}^{i}_s + \int_0^t  (\bar \mu_{s-} ( f) )^{1 / \alpha}  d S^\alpha_s ,\; t \geq 0,   i \in \N^* ,
\end{equation}
with $ \bar  \mu_s = {\mathcal L } ( \bar X^i_s | S^\alpha_u, u \le s ) .$ In the case $ 1 < \alpha < 2, $ we exclude main jumps; that is, we suppose that $ \psi ( \cdot ) \equiv 0 .$ This is due to the fact that the stochastic integral with respect to 
$ \bar Z^i $ has to be treated in the $L^1- $ norm which is a norm not suited for the integral with respect to $ S^\alpha ,$ see Remark \ref{rem:bigjumps} below.

The present paper establishes the proof of the strong convergence of the finite system to the limit system, with respect to a convenient distance. This is done by proposing a coupling of the finite system with the limit one. More precisely we will construct a particular version $S^{N,\alpha}$ of the stable process which is defined on an extension of the same probability space on which the finite system $ X^N $ is defined, and then we consider the limit system driven by $ S^{N, \alpha}.$  So we have to ensure first the existence of a unique strong solution of \eqref{eq:dynlimintrovrai}. 
This is relatively straightforward, when considering the system before the first big jump of the driving stable process, bigger than $K$ for some fixed large $K$. This strategy is inspired by the study of classical SDEs driven by L\'evy noise proposed by \cite{fournier}, it has been used in the framework of non-conditional McKean-Vlasov equations by  \cite{cavallazzi23},  and it can be very easily extended to the present framework of conditional McKean-Vlasov equations. 

In a second step, we then prove the strong convergence of the finite system to the limit system.  
The main ingredient of this step is an explicit construction of the stable process $ S^{ N, \alpha} $ based on the random nearly stable heights of the collateral jumps present in the finite particle system. We discretize time and freeze the jump rate during small time intervals of length $ \delta  . $ Let us suppose for simplicity that $ \nu $ is already the law of a strictly stable random variable (this assumption will not be needed in the sequel of the text). Then the total contribution of collateral jumps during one such interval is a random sum, renormalized by $ N^{ 1 /\alpha},$ constituted of independent stable random variables, each representing one collateral jump. The total number of terms  in the sum is Poisson distributed (conditionally). Thus we are able to use the self-similarity property of the $ \alpha-$stable law: we know that, if $ Y_n$ are i.i.d. strictly  $ \alpha-$stable random variables, then 
\begin{equation*}
 Y_1 +\ldots + Y_n \sim n^{ 1 / \alpha} Y_1 \footnote{Since we need this exact self-similarity property of the $ \alpha-$stable law we are not able to truncate the stable random variables $ Y_n$ as it is often done in the literature.}.
\end{equation*}  
The following trivial but very useful result says that this property survives for random sums, and it is the main argument of our coupling construction.

\begin{proposition}\label{prop:scaling}
Let $ Y_n$ be i.i.d. strictly $ \alpha-$stable random variables. Let $P$ be an integer-valued random variable, independent of $(Y_n)_n.$ Then the following equality holds.
\[\sum_{n=1}^P Y_{n} = P^{1/\alpha} \tilde Y_1 ,\; \; \;  \tilde Y_1 \sim Y_1,\]
where $P$ and $\tilde Y_1$ are independent. 
\end{proposition}

We apply this result on each time interval $ [ k \delta, (k+1) \delta [,$ $ k \geq 0,$ such that the total number of jumps during this time interval, $ P,$ follows a Poisson distribution with parameter $ \delta \sum_{i=1}^N f (x^i )$, conditionally to $\mathcal{F}_{k\delta}$. Here, $ f(x^i) $ is the frozen jump rate of particle $i$ at time $k \delta.$ We then use the law of large numbers for the Poisson random variable $P$ to replace $P$ by its intensity $ \delta  \sum_{i=1}^N f (x^i ) =  N  \delta \int f d \mu^{N,x} $ (where $ \mu^{N, x} $ is the empirical measure $N^{-1} \sum \delta_{x^i} $), so that  $N^{ - 1 / \alpha}  P^{1/\alpha}  \sim (\delta \int f d\mu^{N,x})^{1/\alpha} $. So the contribution of collateral jumps, over one time interval, is approximately given by $ ( \int f d \mu^{N,x} )^{1/\alpha}  \delta^{1 /\alpha } \tilde Y_1,$ and since $ \delta^{1 /\alpha } \tilde Y_1 \sim S^\alpha_\delta, $ it is then reasonable to expect that, as $ \delta \to 0, $ the joint contribution of all small intervals gives rise to the stochastic integral term $  \int_0^t  (\bar \mu_{s-} ( f) )^{1 / \alpha}  d S^\alpha_s$ that appears in \eqref{eq:dynlimintrovrai} above, and this is precisely the strategy of our proof, up to technical details.

Our main result Theorem \ref{thm:prop:chaos}  then shows that for all $ t > 0, $ for all $ i = 1, \ldots , N , $ 
in the case $\alpha>1$,
$$
 \lim_{ N \to \infty } \E[  | X^{N,i}_t -  \bar X^{i}_t | ] = 0 ,
$$
and, in the case $\alpha<1$, for any $0< \am < \alpha$,
$$ \lim_{ N \to \infty } \E[  | X^{N,i}_t -  \bar X^{i}_t | \wedge  | X^{N,i}_t -  \bar X^{i}_t|^{ \am }  ] =  0  .
 $$  
\sloppy Moreover, for all $ K > 0$ a precise control on $\E[  \indiq_{\{t < T^N_K\}} | X^{N,i}_t -  \bar X^{i}_t | ] $ in the case $\alpha>1$ and $\E[ \indiq_{\{t < T^N_K\}} | X^{N,i}_t -  \bar X^{i}_t | \wedge  | X^{N,i}_t -  \bar X^{i}_t|^{ \am }  ]$ in the case $0<\alpha<1$ is given 
in terms of $ \alpha, N , t $ and of the truncation level $ K, $ where $T^N_K$ is the first moment when  the driving $\alpha$-stable process in \eqref{eq:dynlimintrovrai} executes a jump exceeding  $K$ (see Theorem \ref{thm:prop:chaos} below for the precise form of this control).

\section*{Comments on the norms we have used} 
Our error bounds contain three main contributions, and each of these contributions has to be treated using a different norm. The final strong error is then obtained by balancing all these error terms. 
\begin{enumerate}
\item The error due to the stable central limit theorem is treated using the $L^1 -$norm in case $ \alpha > 1 ,$ and the norm induced by $ | \cdot | \wedge  |\cdot|^\am $ in case $ \alpha < 1.$ 
\item We need to study stochastic integral terms of the type \sloppy $\int_0^t  ( \mu^N_{s-} ( f) )^{1 / \alpha}  d S^{N, \alpha}_s ,$  with $\mu^N$ the empirical measure of the finite particle system $ (X^{N, i })_{ 1 \le i \le N}, $ $\int_0^t  (\bar \mu_{s-} ( f) )^{1 / \alpha}  d S^{N, \alpha}_s $ and their convergence. Since small and big jumps of $S^{N,\alpha}$ are not integrable in the same norm, as usual, we cut big jumps and work on the event $  \{t < T^N_K\} , $ for some fixed $K.$ To control the dependency on $K$ in the case $\alpha >1 $ we then work in the $ L^\ap-$norm, for some $\alpha <  \ap < 2  .$ For technical reasons, this trick does not work in the case $\alpha< 1$, such that we work with the $L^1- $norm then. 
\item Finally, most error terms that do only concern the finite particle system (such as time discretization errors) have to be controlled in the $L^\am-$norm, since $ X^{N, i }_ t \in L^\am $ does not belong to $ L^\alpha .$ 
\end{enumerate}

\section*{Bibliographical comments}
The property of \textit{conditional propagation of chaos} is related to the existence of a common noise in the limit system and has been a lot studied in the literature; see for instance \cite{carmona_mean_2016}, \cite{coghi_propagation_2016} and \cite{dermoune_propagation_2003}. In these papers the common noise, which is most often a common, maybe infinite dimensional,  Brownian motion, is already present at the level of the finite particle system, the mean field interactions act on the drift of each particle, and the scaling is the classical one in~$N^{-1}.$ On the contrary to this, in our model, the common noise is only present in the limit, and it is created by the $\alpha-$stable limit theorem as a consequence of the joint action of the collateral jumps of the finite particle system and of the scaling in $N^{ - 1/\alpha }.$ 

In the classical setting of \textit{unconditional propagation of chaos}, mean field limits for particle systems driven by general L\'evy noise have also been extensively studied. We refer the interested reader to the 
paper of Graham \cite{graham92} who considers equations driven by (possibly compensated) Poisson random measures. He works under Lipschitz and integrability conditions with the $L^1$-norm  - which is not possible when the driving noise is an $\alpha-$stable process with $ \alpha < 1$. \cite{jourdainetal} have worked with general L\'evy noise, but mostly under $L^2$-conditions. They consider also the case when the driving process does only possess moments of order $ \alpha < 1, $ and in that situation they are only able to obtain weak existence of the limit without proving uniqueness in law. In the recent paper \cite{cavallazzi}, the author obtains a quantitative propagation of chaos result for systems driven by $ \alpha-$stable subordinators, in the case $ \alpha \in ]1, 2[ .$ There is no measure dependent term within the stochastic integral term in \cite{cavallazzi}, and the author is mostly interested in relaxing the regularity assumptions on the coefficients and works only under the assumption of H\"older continuity. Finally, several papers are devoted to the well-posedness of the limit equation. Let us mention \cite{menozzietalkonakov} which treats the general case $ 0 < \alpha < 2, $ under mild regularity assumptions, by the means of an associated non-linear martingale problem. Finally,  \cite{cavallazzi23} treats the case $ 1 < \alpha < 2, $ including a measure dependent term in the stochastic integral term, under general Lipschitz assumptions. All these papers are devoted to the unconditional framework.    

Let us come back to the discussion of our present paper, which deals with the conditional propagation of chaos property in a situation where the driving L\'evy process appears only in the limit system. We have already mentioned that it continues and extends the diffusive setting studied in \cite{ELL1} and \cite{ELL2} to the framework of the $ \alpha-$ stable limit theorem. As in  \cite{ELL1} and \cite{ELL2}, the basic strategy  is to discretize time and construct a coupling with the driving noise of the limit equation explicitly within each time interval $ [k \delta, (k+1)\delta) .$ This strategy is actually inspired by the approach proposed in \cite{DGLP}. 

We now quickly discuss the main differences with respect to the former diffusive setting. In the bounded variation regime $ 0 < \alpha < 1,$ the error due to time discretization is now of order $ \delta ,$  as opposed to $ \sqrt{\delta} $ which was the leading order in the diffusive scaling. Contrarily to the diffusive case,  the main contribution to the error comes now from the error made when replacing the Poisson variable $P$ by its expectation and from the quantified error in the stable limit theorem. The first error gives rise to a term of order $ ( N \delta)^{ - \alpha/2} ,$ and this is a main contribution to the error. In the diffusive case this error was negligible. Concerning the quantified rate of convergence in the stable limit theorem, we suppose that the law of the collateral jumps is heavy-tailed and rely on recent results of \cite{chen22}. Finally, there is the error which is due to the bound on the Wasserstein$-\am$ distance between the empirical measure $ \bar \mu^N_s$ of the limit system and $ \bar \mu_s .$ We rely on results obtained by \cite{fournierguillin} to control this error.  

In case $ 1 < \alpha < 2, $ to control the error is more complicated than in the original diffusive case. Indeed, the presence of the stochastic integral imposes that we have to deal with the small jumps of the driving stable process with respect to at least the $L^\ap-$norm. And this is what we do. However, all errors related to the finite particle system can only be treated with respect to the $L^\am-$norm. So we use H\"older's inequality repeatedly -- and each time we do this, we loose with respect to the original convergence rate. The error due to time discretization is now of order $ \delta^{1 / \am} ,$ which gives another important contribution.  Concerning the quantified rate of convergence in the stable limit theorem, we rely on recent results of \cite{xu}.

Our paper is accompanied by a companion paper \cite{dasha-eva} which studies the case $ 0 < \alpha < 1, $ in a particular framework where all jumps are positive and where big jumps do not need to be cut.

\section*{General notation}
Throughout this paper we shall use the following notation. Given any measurable space $ (S, \mathcal S), $ $\mathcal{P} (S)$ denotes the set of all probability measures on $ (S, \mathcal S),$ endowed with the topology of weak convergence. For $p> 0,$ $\mathcal{P}_p(\R)$ denotes the set of probability measures on $\R$ that have a finite moment of order~$p$.  
For two probability measures $\nu_1, \nu_2 \in \mathcal{P}_p (\R),$ the Wasserstein distance of order $p$ between $\nu_1$ and $\nu_2$  is defined as
\[
W_p(\nu_1,\nu_2)=\inf_{\pi\in\Pi(\nu_1,\nu_2)}\left( \int_\R\int_\R |x-y|^p \pi(dx,dy) \right)^{(1/p) \wedge 1 } ,
\]
where $\pi$ varies over the set $\Pi(\nu_1,\nu_2)$ of all probability measures on the product space $\R\times \R$ with marginals $\nu_1$ and $\nu_2$. Notice that  the Wasserstein distance of order $p$ between $\nu_1$ and $\nu_2$ can be rewritten as  the infimum of $(\E[| X - Y|^p])^{(1/p) \wedge 1 }$ over all possible couplings $(X,Y)$ of the random elements $X$ and $Y$ distributed according to $\nu_1$ and $\nu_2$ respectively, i.e.
\[
W_p(\nu_1,\nu_2)=\inf\left\{(\E{|X - Y|^p})^{(1/p)\wedge 1 }: \mathcal{L}(X)=\nu_1\ \mbox{and} \ \mathcal{L}(Y)=\nu_2  \right\}.
\]
Moreover, the Kantorovitch-Rubinstein duality yields
\[ W_1(\nu_1,\nu_2) = \sup\{ \nu_1(\varphi) - \nu_2(\varphi): \forall\, x,\,y\in \R\,\, |\varphi(x)-\varphi(y)|\leq  |x-y|\} .\]  
Furthermore, for any $q \leq 1$, and for all $ x, y \in \R, $ we write
\begin{equation}\label{eq:dq}
d_{q}(x,y) = |x-y| \wedge |x-y|^{q} \mbox{ and } \| x \|_{d_q}  := |x| \wedge |x|^q .
\end{equation}
We frequently use the fact that, if restricted to the positive half line,  $\R_+ \ni x \mapsto\|x\|_{d_q} $ is increasing and concave, and therefore sub-additive, that is,  
\begin{equation}\label{eq:dqsubadditive}
\|x+y\|_{d_q}\leq \|x\|_{d_q} + \| y \|_{ d_q} , \mbox{ for all } x, y \geq 0. 
\end{equation}
In particular, $ d_q (x, y ) $ defines a metric on $ \R.$

Finally, for any $\nu_1,\, \nu_2 \in \mathcal{P}_1(\R)$, 
\[ W_{d_{q}}(\nu_1,\nu_2) = \inf_{\pi \in \Pi(\nu_1,\nu_2)}\left(\int_{\R}\int_{\R} d_{q}(x,y) \pi(dx,dy) \right)\]
is the Wasserstein distance associated with the metric $d_{q}$ (see \cite{chen22}). 
As for the classical Wasserstein distance, the Kantorovitch-Rubinstein duality yields (\cite{villani}, particular case 5.16)
\[W_{d_{q}}(\nu_1,\nu_2) = \sup\{\nu_1(\varphi) - \nu_2(\varphi) : \forall x,\,y\in \R\,\, |\varphi(x) - \varphi(y)| \leq  d_{q}(x,y) \} . \]
Notice that $W_{d_{q}}(\nu_1,\nu_2)\leq W_1(\nu_1,\nu_2)$ for all $\nu_1,\,\nu_2 \in \mathcal{P}_1(\R)$. 

Moreover, $D(\R_+,\R)$ denotes the space of c\`adl\`ag functions from $\R_+$ to $\R $, endowed with the Skorokhod metric, and 
$C$ and $K$ denote arbitrary positive constants whose values can change from line to line in an equation. We write $C_\theta$ and $K_\theta$ if the constants depend on some parameter $\theta.$

Finally, throughout this paper, $ \alpha_-  < \alpha$ and $\alpha_+ > \alpha $ will be two fixed constants belonging to $ ( 0, 2 ) ,$ one strictly smaller and the other strictly larger than $ \alpha ,$ the index of the driving stable process.  
We also suppose that $ \am > 1 $ in case $ \alpha > 1.$

\section{
 Model, assumptions, main results, organisation of the paper}\label{sec:model}
 \subsection{The model}\label{subsec:model}
Throughout this article, $ S^\alpha = (S^\alpha_t)_{t \geq 0} $ denotes an  $\alpha-$stable L{\'e}vy process given by (\cite{applebaum}, \cite{sato})
\begin{align}\label{eq:St}
 & S^\alpha_t = \int_{ [0, t ] \times \R^\ast} z \tilde M (ds, dz ) , \mbox{ if } \alpha >  1 \\ \nonumber
& S^\alpha_t = \int_{ [0, t ] \times \R^\ast} z M (ds, dz ) , \mbox{ if } \alpha < 1 .
  \end{align}
Its jump measure $M$ is a Poisson random measure on $ \R_+ \times \R^\ast $ having intensity $ds \nu^\alpha(dz)$, with
\begin{equation*}
\nu^\alpha(dz) = \frac{a_+}{z^{\alpha+1}}\indiq_{\{z>0\}} dz + \frac{a_-}{|z|^{\alpha+1}}\indiq_{\{z<0\}}dz ,
\end{equation*}
where $ a_+, a_- \geq 0 $ are some fixed parameters, and $\tilde M(ds,dz) \coloneqq M(ds,dz) - \nu^\alpha(dz) ds$ denotes the compensated Poisson random measure.

In what follows we will consider random variables which are distributed according to a heavy-tailed law which belongs to the strong domain of attraction of a stable law, according to the following definition. 

\begin{definition}\label{def:strongdomain}
Following Example 2 in \cite{xu} and \cite{chen22}, we say that a law is heavy-tailed with indices $\alpha, \gamma, \beta, A $ and $ \tilde A$, with $ 0 \le \alpha < 2, \alpha \neq 1, \gamma > 0 ,$  if its distribution function $G$ has the form
\begin{equation*}
\begin{split}
 1- G(x) & = \frac{A}{|x|^\alpha} (1+\beta) + \frac{\tilde A}{|x|^{\alpha+\gamma}}(1+\beta), \qquad \qquad x\geq L, \\
 G(x) &  = \frac{A}{|x|^\alpha} (1-\beta) + \frac{\tilde A}{|x|^{\alpha+\gamma}}(1-\beta), \qquad\qquad x \leq -L,
\end{split}
\end{equation*}
for some $L>0$, 
where $\beta \in [-1,1]$ encodes the asymmetry in the distribution, $A, \tilde A >0$ are such that $|L|^{-\alpha} (A+ |L|^{-\gamma}\tilde A) \leq \frac{1}{2}$, and $\gamma >0$. 
In particular, such a law belongs to the domain of attraction of an $\alpha$-stable law (see \cite{feller}, IX.8, Theorem 1).  
\end{definition}
Following Chapter IX, Section 8 of \cite{feller} and Section 14 of \cite{sato}, one finds that the values of the parameters $A$ and $\beta$ appearing in Definition \ref{def:strongdomain} are related to the parameters $a_+ $ and $a_- $ of the L{\'e}vy measure $\nu^\alpha$ of the limit stable process by $a_+ = (1+\beta) \alpha A$ and $a_- =  (1- \beta) \alpha A$.
\begin{remark}
We can easily see that the $q$-th absolute moments of a law satisfying Definition \ref{def:strongdomain} are finite for $q<\alpha$ (see also Theorem 3 in \cite{GK}, Part III, Chapter 7, Section 35 for a more general result concerning the domain of attraction of stable laws). This will be often employed in the sequel.    
\end{remark}

After these preliminary definitions, we now introduce our finite particle system. To do so, let $(\pi^{i}(ds,dz,du))_{i\geq 1}$ be a family of i.i.d. Poisson measures on $\R_+\times \R_+\times \R$ having intensity measure $dsdz\nu(du)$, where, in the case $\alpha<1$, $\nu $ satisfies Definition \ref{def:strongdomain} with parameters $\alpha , \gamma, \beta, A , \tilde A$, and, in the case $\alpha >1$, $\nu = \mathcal L (\xi  - \E (\xi) ), $ where $ \xi $ is a real-valued random variable with distribution function $G$ satisfying Definition \ref{def:strongdomain} with parameters $\alpha , \gamma, \beta, A , \tilde A$. 

Consider also an i.i.d. family $(X_0^{i})_{i\geq 1}$ of $\R$-valued random variables independent of the Poisson measures, distributed according to some fixed probability measure $\nu_0 $ on $ (\R, {\mathcal B} ( \R) ).$  
In what follows we write $ ( \Omega, {\mathcal A}, {\mathbf P} ) $ for the basic probability space on which are defined all $ \pi^i $ and all $ X_0^i, $ and we  use the associated canonical filtration  
\begin{equation*}\label{eq:filtration}
 {\mathcal F}_t = \sigma \{  \pi^i  ( [0, s ] \times A\times B ), s \le  t , A \in {\mathcal B} ( \R_+ ) , B \in {\mathcal B} ( \R ) , i \geq 1 \} \vee \sigma \{ X^{ i }_0, i \geq 1 \}, \; t \geq 0.
\end{equation*}  
We will also use the projected Poisson random measures which are defined by 
\begin{equation*}\label{eq:projpi}
 \bar \pi^i (ds, dz) = \pi^i ( ds, dz, \R) , 
\end{equation*} 
having intensity $ds  dz .$ For any $ N \in \N^\ast, $ we consider a system of interacting particles  $ (X^{N, i}_{t}) ,  t\geq 0, \;  1 \le i \le N , $ evolving according to 
\begin{multline}\label{eq:micro:dyn}
X^{N, i}_t =  X^{i}_0 +   \int_0^t  b(X^{N, i}_s , \mu_s^N )  ds +  \int_{[0,t]\times\R_+} \psi (X^{N, i}_{s-}, \mu_{s-}^N )  \indiq_{ \{ z \le  f ( X^{N, i}_{s-}) \}} \bar\pi^i (ds,dz) \\
+ \frac{1}{{N}^{ 1/\alpha} }\sum_{ j \neq i } \int_{[0,t]\times\R_+\times\R  }u \indiq_{ \{ z \le  f ( X^{N, j}_{s-}) \}} \pi^j (ds,dz,du),
\end{multline} 
where 
$ \mu_t^N = \frac1N \sum_{i=1}^N \delta_{X_t^{N, i } } $
is the empirical measure of the system at time $t.$ In the above equation, $ b : \R \times {\mathcal P}_1 ( \R) \to \R ,$  $ \psi : \R  \times {\mathcal P}_1(\R)\to \R  $ and $ f : \R \to \R_+$ are measurable and bounded functions. In case $ \alpha > 1, $ we always suppose that $ \psi (\cdot ) \equiv 0, $ that is, there are no main jumps.

\begin{remark}\label{rem:finite_sum}
Notice that, since for each $i ,$ $\pi^i([0,t]\times[0,\|f\|_\infty]\times\R)$ is Poisson distributed with parameter $\|f\|_\infty t\times\nu(\R)=\|f\|_\infty t,$ the number of atoms of the measures $\pi^i$s in $[0,t]\times[0,\|f\|_\infty]\times\R $ is a.s. finite. Hence the integral in the r.h.s. of \eqref{eq:micro:dyn} is in fact a sum with a.s. a finite number of terms. 
\end{remark}

In what follows we will provide additional conditions on the functions $b, f $ and $ \psi $ that, together with our preliminary considerations, in particular Proposition \ref{prop:scaling}, allow to show that, as $ N \to \infty, $ the above particle system converges (in law) to an infinite exchangeable system $ (\bar X^i )_{i \geq 1 } $ solving
\begin{multline}\label{eq:limitsystem}
\bar X^{i}_t =  X^{i}_0 +   \int_0^t  b(\bar X^{ i}_s , \bar \mu_s )  ds +  \int_{[0,t]\times\R_+ } \psi (\bar X^{ i}_{s-}, \bar \mu_{s-} )  \indiq_{ \{ z \le  f ( \bar X^{ i}_{s-}) \}} \bar \pi^i (ds,dz) 
\\
+ \int_{[0,t] } \left( \bar \mu_{s-}(f)  \right)^{1/ \alpha}   d S_s^\alpha , 
\end{multline} 
where $S^\alpha,$ given by \eqref{eq:St}, is independent of the collection of Poisson random measures $ (\bar \pi^i (ds, dz))_{ i \geq 1 }$ and of the initial values $ (X^i_0)_{i \geq 1}$, and where $ \bar \mu_s = {\mathcal L} ( \bar X^1_s | S^\alpha_u, u \le s ). $

The main part of this article is devoted to the proof of (a quantified version of) the convergence of the finite system \eqref{eq:micro:dyn} to the limit system \eqref{eq:limitsystem}.

But, before doing so, we briefly discuss strong existence and uniqueness of the particle system and its associated limit system.

\subsection{Assumptions} 
To prove well-posedness of the particle and limit systems, we will only need the following Assumptions \ref{ass:b}--\ref{ass:nu0}:
\begin{assumption}\label{ass:b}
\begin{enumerate}[label=\alph*)]
\item $b$ is  bounded. \label{itm:bbdd}
\item There exists a constant $ C > 0$ such that for every $x,\,y \in \R$ and every $\mu,\,\tilde \mu \in \mathcal{P}_1(\R)$, it holds $ |b(x,\mu) - b(y,\tilde \mu) | \leq C \left( |x-y| +W_{1}(\mu,\tilde \mu )\right).$ \label{itm:bLip} 
\end{enumerate}
\end{assumption}

\begin{assumption}\label{ass:f}
\begin{enumerate}[label=\alph*)]
\item $f$ is lowerbounded by some strictly positive constant $ \underline f > 0$. \label{itm:flbdd} 

\item $f$ is bounded. \label{itm:fbdd}

\item $f$ is Lipschitz-continuous.  \label{itm:fLip}
\end{enumerate}
\end{assumption}

Recall that we assumed that there are no main jumps in case $ \alpha > 1.$ In case $ \alpha < 1, $ to deal with the main jumps, we also suppose that

\begin{assumption}\label{ass:psi}
\begin{enumerate}[label=\alph*)]
\item $ \psi$ is bounded. \label{itm:psibdd}
\item  There exists a constant $ C > 0$ such that for every $x,\,y \in \R$ and every $\mu,\,\tilde \mu \in \mathcal{P}_1(\R)$, it holds $ |\psi(x,\mu) - \psi(y,\tilde \mu) | \leq C \left( |x-y|  +W_{1}(\mu,\tilde \mu)\right).$ \label{itm:psilip}
\end{enumerate}
\end{assumption}

Recall that the initial positions $ (X^i_0)_{ i \geq 1}  $ are i.i.d., distributed according to some fixed probability measure $ \nu_0.$ We assume:
\begin{assumption}\label{ass:nu0}
$\nu_0 $ admits a finite first moment in case $ \alpha < 1 $  and a finite second moment in case  $ 1 < \alpha < 2. $
\end{assumption}

 To prove the convergence to the limit system we also need the Assumptions \ref{ass:f+}--\ref{ass:nu0+}: 
\begin{assumption}\label{ass:f+}
 $f\in{\cal C}^1$. 
\end{assumption}

In case $ \alpha < 1, $ we have to strengthen the Lipschitz assumptions \ref{ass:b}\ref{itm:bLip} and \ref{ass:psi}\ref{itm:psilip} and suppose additionally that 
\begin{assumption}\label{ass:bqLipp+psiqLip}
There exists a constant $ C > 0$ such that for some fixed  $ \am \in (0, \alpha ), $ for every $x,\,y \in \R$ and every $\mu,\,\tilde \mu \in \mathcal{P}_1(\R)$, $ |b(x,\mu) - b(y,\tilde \mu) | + |\psi(x,\mu) - \psi(y,\tilde \mu) | \leq C \left( d_{ \am }(x,y) +W_{d_{ \am}}(\mu,\tilde \mu)\right).$ \label{itm:bqLip}
\end{assumption}

In both cases we assume 
\begin{assumption}\label{ass:nu0+}
\begin{enumerate}[label=\alph*)]
\item
The initial distribution $\nu_0$ of each coordinate $X^i_0$ in \eqref{eq:micro:dyn} and \eqref{eq:limitsystem} admits a finite moment of order $(2\alpha) \vee 1$ if $\alpha<1$ and a finite moment of order $p$ for some $p>2$ if $1<\alpha<2$. 
\item  $\nu $ is heavy-tailed according to Definition \ref{def:strongdomain}, for some $ \alpha \in (0, 2 ) \setminus \{1\}$ and for some $ \gamma $ such that $ \alpha + \gamma \notin \{1, 2\}.$ Furthermore, $ \nu$ is centered if $ \alpha > 1$, that is, for $\alpha>1,$ $\nu = \mathcal L (\xi  - \E (\xi) ), $ where $ \xi $ is a real-valued random variable with distribution function $G$ as in Definition \ref{def:strongdomain}. 
\end{enumerate}
\end{assumption}

\begin{remark}
In what follows we will repeatedly use that Assumption \ref{ass:f} implies that for any two probability measures $ \mu , \tilde\mu  \in \mathcal{P}_1(\R) $ and for any $ p \geq 1,$ 
\begin{equation}\label{eq:controlmuf}
|( \mu ( f))^{1 / \alpha } - ( \tilde\mu ( f) )^{1/\alpha } |^p  \le C | \mu ( f) - \tilde\mu ( f) |^p  \le C  | \mu ( f) - \tilde\mu ( f) | \le C W_1 ( \mu , \tilde\mu ) , 
\end{equation}
where $C$ is a constant that may change from one occurrence to another. Here we used that $ z \mapsto z^{ 1/ \alpha } $ is Lipschitz on $ [ 0, \|f\|_\infty ] $ in case $ \alpha < 1 $ (on $ [ \underline f , \infty ] $ in case $ \alpha > 1$) together with the boundedness of $f.$ 
In the case $\alpha < 1 , $ we will also use that, for any $\alpha_- < 1$, 
\begin{equation}\label{eq:controlmuf2}
|( \mu ( f))^{1 / \alpha } - ( \tilde\mu ( f) )^{1/\alpha } |^p \le C | \mu ( f) - \tilde\mu ( f) |^p  \le C  | \mu ( f) - \tilde\mu ( f) | \le C W_{d_{\alpha_-}} ( \mu, \tilde\mu ) .
\end{equation}
This is derived similarly to \eqref{eq:controlmuf} and observing that the boundedness and Lipschitz continuity of $f$ imply that there exists $C>0$ such that for every $x,\,y\in \R$ it holds $|f(x) - f(y) |\leq C d_{\alpha_-}(x,y)$, so we can employ the Kantorovitch-Rubinstein duality on $f/C$ to conclude.  
\end{remark}

\subsection{Main results}\label{subsec:mainresults}
\begin{theorem}\label{thm:existfinite}
    Grant Assumptions  \ref{ass:b}\ref{itm:bLip}, \ref{ass:f}\ref{itm:fbdd}, \ref{ass:psi}  and \ref{ass:nu0}. 
    Then system \eqref{eq:micro:dyn} admits a unique strong solution. Furthermore, for all $N\in \N^\ast$, $i= 1,\ldots, N$ and $t>0$, $X^{N,i}_t$ has a finite moment of order $p$ for all $p<\alpha$. 
\end{theorem}
The proof of Theorem \ref{thm:existfinite} is given in the Appendix section \ref{subsec:strongexistfin}. 
\begin{theorem}\label{theo:strongexistence}
Grant Assumptions \ref{ass:b}, \ref{ass:f}, \ref{ass:psi} and \ref{ass:nu0}. Then \eqref{eq:limitsystem} 
admits a unique strong solution. Furthermore, for all $t>0$, $\bar X_t$ has a finite moment of order $p$ for all $p<\alpha$. 
\end{theorem}
The proof of Theorem \ref{theo:strongexistence} is given in Section \ref{sec:strongexist}. 

We may now state our main theorem. Remember that $ ( \Omega, {\mathcal A}, {\mathbf P} ) $ denotes the space on which all the $\pi_i$ and $X_0^{N,i}, i\in\N^*, N\in\N^*, $ are defined. Let $(X^{N, i})_{ 1 \le i \le N} $ be the unique strong solution of \eqref{eq:micro:dyn} driven by $(\pi_i)_{ 1 \le i \le N} $.

\begin{theorem}\label{thm:prop:chaos}
Grant Assumptions \ref{ass:b}--\ref{ass:nu0+}. Then the following holds for all  $t > 0 $ and $ i \geq 1.$ 
\begin{enumerate}
\item
If $ \alpha > 1, $ then 
$$ \lim_{N \to \infty} W_1  ( {\mathcal L}( X_t^{N, i } ) , {\mathcal L} ( \bar X_t^i ) ) \to 0 ,$$
and, if $ \alpha_{-} <  \alpha < 1, $ then
$$  \lim_{N \to \infty}  W_{d_{\alpha_{-} }} ( {\mathcal L}( X_t^{N, i } ) , {\mathcal L} ( \bar X_t^i ) ) \to 0 .$$

\item 
Moreover, for any $N\in \N^\ast, \delta\in (0,1)$ such that $2 \delta \norme{f}_\infty < 1,$ we can construct, on an extension of $(\Omega,\mathcal{A},\mathbf{P}),$ a one-dimensional strictly $\alpha$-stable process $S^{N,\alpha,\delta}$, independent of the initial positions $(X^i_0)_{i=1,\dots,N}$ and of $ (\bar \pi^i)_{i \geq 1},$ such that the following holds. 

If  $(\bar X^{N, \delta , i})_{ i \geq 1 } $ denotes the unique strong solution of the limit system \eqref{eq:limitsystem} driven by $S^{N,\alpha,\delta}$ and  
$ (\bar \pi^i)_{ i \geq 1},$  and writing $T^N_K \coloneqq \inf{\{t \geq 0 \, : \, |\Delta S^{N,\alpha , \delta}_t | > K \}}$, we have % for all $t\geq 0,$ $i=1,\dots,N$ and 
for all $ K > 0, $
\begin{enumerate}
\item for any $1<\alpha_- <\alpha<\ap<2,$ 
\begin{equation}\label{eq:finrate:alpha>1}
\E[\indiq_{\{t < T^N_K\}}| X^{N,i}_t - \bar X^{N,\delta, i}_t|] \leq  e^{C t \frac{K^{\ap-\alpha}}{\ap - \alpha}}  \left(  \left(N^{1-\frac{\am}{\alpha}}\delta \right)^{\frac{1}{\ap}} + r_t(N,\delta) + N^{-1/2} \right)^{1/\ap}, 
\end{equation}
where 
\begin{equation*}\label{eq:r(Ndelta):alpha>1}
r_t(N, \delta) \coloneqq N^{ \frac{1}{\am} \left( 1 + \frac{1}{\am}\right) \left( 1 - \frac{\am}{\alpha}\right) } \delta^{\frac{1}{ (\am)^2}} 
 + \lceil\frac{t}{\delta}\rceil  \delta^{\frac{1}{\alpha}}\left(   g(N \delta)+  (N\delta)^{- \frac{1}{2}} \right)  + \delta^{\frac{1  }{\alpha}} ,
\end{equation*}
and where the function $g$ is given in \eqref{eq:def:g} below;
\item for any $0<\am<\alpha<1,$
\begin{multline}\label{eq:finrate:alpha<1}
\E[\indiq_{\{t < T^N_K\}}  d_{\am}( X^{N,i}_t , \bar X^{N,\delta, i}_t) ] 
\le 
e^{C t \frac{K^{1-\alpha}}{1 - \alpha} } \times \\
\times  \left(  r_t(N, \delta)  +N^{ - 1/2} \indiq_{ \{1 > \am > \frac12\}} +
N^{- \am }\indiq_{\{ \am < \frac12\}}\right) ,
\end{multline} 
where 
\begin{equation*}\label{eq:r(Ndelta):alpha<1}
r_t(N,\delta) \coloneqq  N^{ 2\left(1 - \frac{\am}{\alpha}\right)}   \delta 
 + \lceil\frac{t}{\delta}\rceil \delta^{\frac{\am}{\alpha}} \left(   g(N \delta)+  (N \delta)^{-\frac{\am}{2}}   \right) + \delta^{\frac{\am}{\alpha}} , 
\end{equation*}
and where the function $g$ is given in \eqref{eq:def:g:alpha<1} below. 
\end{enumerate}
\end{enumerate}
\end{theorem}

\begin{remark}
%Let us briefly discuss the implications of item 1. of the above result. 
In case $ \alpha > 1, $ item 1. of the above theorem implies the weak convergence of $ {\mathcal L}( X_t^{N, i } )  $ to $  {\mathcal L} ( \bar X_t^i ) ,$ as $ N \to \infty, $ together with the convergence of the first moments (Theorem 6.9 of \cite{villani}). 
In case $ \alpha < 1, $ $ W_{d_{\alpha_{-}  }} $ is the Wasserstein-$1$ distance associated with the distance $ d_{\alpha_{-}}, $ see Definition 6.1 in \cite{villani}. 
Since continuity with respect to $d_{ \alpha_{-} } $ is equivalent to continuity with respect to the usual Euclidean distance $d_1, $ 
Theorem 6.9 of \cite{villani} together with item 1. imply the weak convergence of $ {\mathcal L}( X_t^{N, i } )  $ to $  {\mathcal L} ( \bar X_t^i ) ,$ as $ N \to \infty, $ together with the convergence of the first $d_{\alpha_{-}}-$moments.
\end{remark}

In the following we will drop the dependence of $S^{N,\alpha,\delta}$ and $ \bar X^{N, \delta, i } $ on $\delta$ for convenience of notation, as we will always take $\delta = \delta(N)$.  

\begin{remark}\label{rem:rate}
The rate of convergence stated in the above theorem depends on our choice of $ \am < \alpha ,$ $\ap>\alpha$ and on $ \delta$. To get an idea of the leading term in the error, taking formally $ \am \uparrow \alpha, \, \ap \downarrow\alpha$ and choosing $ \delta = \delta (N)$ such that all error terms are equilibrated, we obtain the following rate of convergence (see Appendix section \ref{sec:delta_choice}). 
\begin{enumerate}
\item For $ \alpha > 1 $ and $ \gamma + \alpha < 2, $ 
\[ N^{ - \frac{\gamma}{ \alpha^2 (1-\alpha+\gamma\alpha + \alpha^2)} } \indiq_{ \gamma < \alpha/2 } + N^{ - \frac{1}{2 \alpha\left(1-\alpha +\frac{3}{2}\alpha^2\right)} } \indiq_{  \gamma \in \left(\frac{\alpha}{2}, 2-\alpha\right) }    .\] 
\item For $ \alpha > 1 $ and  $\gamma + \alpha > 2,$
\[ N^{ - \frac{1}{2 \alpha\left(1-\alpha +\frac{3}{2}\alpha^2\right)} }  \indiq_{ \alpha \in ( 1, \frac43 )  } +N^{-\frac{2-\alpha}{\alpha^2 (1+\alpha)}}  \indiq_{\alpha > \frac{4}{3}} .\]
\item For $\alpha \in (0,\sqrt{3}-1)$ and $\gamma < \frac{\alpha^2}{2}$ or for $\alpha \in (\sqrt{3}-1,1)$ and $\gamma < 1-\alpha,  $
\[ N^{-\frac{\gamma}{\alpha+\gamma}}.\]
\item For $\alpha \in (0,\sqrt{3}-1)$ and $\gamma > \frac{\alpha^2}{2}$, 
\[  N^{-\frac{\alpha}{2+\alpha}} .\]
\item
For $\alpha \in (\sqrt{3}-1,1)$ and $\gamma > 1-\alpha,$ 
\[ N^{\alpha-1}.\]
\end{enumerate}

We discuss in Section \ref{sec:discussionrate} in detail  the different ingredients that constitute our rates in \eqref{eq:finrate:alpha>1} and \eqref{eq:finrate:alpha<1}. 

\end{remark}
\subsection{Organisation of the paper}
The rest of the paper is organized as follows. Section \ref{sec:strongexist} is devoted to the proof of pathwise uniqueness for the limit system. The proofs of well-posedness of the particle system and strong existence for the limit one are postponed to Appendix sections \ref{subsec:strongexistfin} and \ref{app:strong_ex_lim}, since they are based on standard techniques. 

In Section \ref{sec:coupling}, we provide a representation of the interaction term in \eqref{eq:micro:dyn} in terms of a stochastic integral with respect to an $\alpha$-stable process. This representation entails a time discretization of the particle system (see in particular Subsection \ref{subs:Y}) and relies on (a generalization of) the stable CLT and on previously obtained bounds in \cite{xu, chen22} (see Proposition \ref{lem:preliminary_results}, where the errors due to both the CLT and the rates of \cite{xu, chen22} are introduced). The overall error that we make by all these approximations will be controlled using the $L^1$-norm in the case $\alpha\in (1,2)$ and the norm induced by the distance $d_{\am}$ in the case $\alpha\in(0,1)$ (see $R^N_t$ in the statement of Theorem \ref{prop:representation_finite_syst}). In Section \ref{sec:coupling}, we will also provide an explicit construction of the limit process driving system \eqref{eq:limitsystem} (see Subsection \ref{subs:bridge}).  

Section \ref{sec:convergencelim} concludes the convergence proof employing some auxiliary results and intermediate useful representations for the particle and the limit systems.  
In particular, in that section additional error terms to the ones collected in Section \ref{sec:coupling} appear due to the need to approximate the conditional law of the limit system by its empirical measure. These errors are controlled in $L^{\ap}$-norm for $\alpha>1$ and in $L^1$-norm for $\alpha<1$ thanks to the properties of boundedness and Lipschitz continuity of the functions $f$ and $b$, and using results in \cite{fournierguillin}.

\section{Strong existence and uniqueness for the limit system}\label{sec:strongexist} 
Here we prove the well-posedness of the limit system \eqref{eq:limitsystem}. We consider one typical particle $ \bar X_t$ representing the limit system  \eqref{eq:limitsystem}. It  evolves according to
\begin{multline}\label{eq:limitequation1}
\bar X_t =  X_0 +   \int_0^t  b(\bar X_s , \bar \mu_s )  ds +  \int_{[0,t]\times\R_+ } \psi (\bar X_{s-},  \bar \mu_{s-}  )  \indiq_{ \{ z \le  f ( \bar X_{s-}) \}} \bar \pi (ds,dz) \\
+ \int_{[0,t] } \left( \bar \mu_{s-}(f)  \right)^{1/ \alpha}   d S_s^\alpha , 
\end{multline} 
where $X_0 \sim \nu_0, $ $ \bar \mu_s = {\mathcal L} ( \bar X_s | S^\alpha_u, u \le s ) ,$ $ \bar \pi $ is a Poisson random measure on $ \R_+ \times \R_+ $ having intensity $ ds dz,$  and where $ S^\alpha , \bar \pi  $  and $ X_0$ are independent. We use the above representation both for $ \alpha < 1 $ and $ \alpha > 1 $ keeping in mind that in the latter case, main jumps are excluded from our study, that is, $ \psi ( \cdot ) \equiv 0.$ See Remark \ref{rem:bigjumps} below. 
We also use the associated canonical filtration
\[ \bar {\mathcal F}_t \coloneqq \sigma  \{  \bar \pi  ( [0, s ] \times A ), s \le  t , A \in {\mathcal B} ( \R_+ )  \} \vee \sigma \{ X_0\} \vee \sigma \{ S_s^\alpha, s \le t \} , \; t \geq 0 .
\]

\begin{remark}\label{rem:33}
Let $p > 0.$ Notice that despite the presence of the integral against the stable process $ S^\alpha $ in \eqref{eq:limitequation1} above, $\bar \mu_t$ almost surely admits a finite moment of order $p$  for any $t \geq 0,$ whenever $ X_0 $ does so. This follows from the boundedness of $b,f$ and $ \psi, $ since 
\begin{equation}\label{eq:mufin}
  \sup_{ s \le t } 
 | \bar X_s| \le |X_0| + \| b\|_\infty t + \| \psi \|_\infty  \int_{[0,t]\times\R_+ }   \indiq_{ \{ z \le \|f\|_\infty  \}} \bar \pi (ds,dz)  + \sup_{ s \le t} \left| \int_{[0, s ]} \bar \mu_{v-}(f)^{ 1/\alpha} d S_v^\alpha  \right| .
 \end{equation} 
Since $ t \mapsto  \int_{[0,t]\times\R_+ }   \indiq_{ \{ z \le \|f\|_\infty  \}} \bar \pi (ds,dz) $ is a Poisson process of rate $ \|f\|_\infty ,$ possessing all moments, we deduce that 
\begin{equation}\label{eq:mufin2}
  \int_\R |x|^p  \bar\mu_t (dx) =    \E [  | \bar X_t|^p |\; S^\alpha ] \le  C_p (t) \left( \E |X_0|^p + 1  + \sup_{ s \le t} \left| \int_{[0, s ]} \bar \mu_{v-}(f)^{ 1/\alpha} d S_v^\alpha \right|^p \right)  < \infty 
 \end{equation}
almost surely. 
\end{remark}

\subsection{Pathwise uniqueness for the limit system }

\subsubsection{The case $\alpha>1$} 
 
We start discussing the case $ \alpha > 1.$ Fix some $ K > 0.$ By the L\'evy-It{\^o} decomposition (see \cite{applebaum}, Theorem 2.4.16), $S^\alpha$ admits the pathwise representation 
\begin{equation}\label{eq:S^alpha_repr}
    \begin{split}
        S^\alpha_t & = \int_{[0,t]\times B^\ast_K} z \tilde M(ds,dz) + \int_{[0,t]\times B^c_K} z M(ds,dz) - \int_{[0,t]\times B^c_K} z \nu^\alpha(dz) ds,
    \end{split}
\end{equation}
$t\geq 0$, where  $B^\ast_K \coloneqq \{ z \in \R \, : \, |z| \leq K\} \setminus \{0\}$ and where $\tilde M$ denotes the compensated jump measure. 

For any $K>0$ define
\begin{equation*}
    \begin{split}
        T_K \coloneqq \inf{\{t\geq 0\, : \, |\Delta S^\alpha_t|>K \}} ,
    \end{split}
\end{equation*}
that is, the first time the process $S^\alpha_t$ has a jump greater than $K$, $\Delta S^\alpha_t \coloneqq S^\alpha_t - S^\alpha_{t^-}.$   
Notice that for any finite $T,$ $ \lim_{K \to \infty  }   \mathbf{P}(T_K > T) = 1 $.

Consider now two solutions to the limit system \eqref{eq:limitsystem}, $X = (X_t)_{t\geq 0}$ and $\tilde X = (\tilde X_t)_{t\geq 0}$, with the same initial condition $X_0 = \tilde X_0$, and denote by $\mu_t$ and $\tilde \mu_t$ the conditional laws of $X_t$ and $\tilde X_t$ respectively given $S^\alpha$.  
Observing that on $\{t < T_K\}, $ the stochastic integral term corresponding to big jumps in \eqref{eq:S^alpha_repr} equals zero, we have 
\begin{multline*}
           \indiq_{\{T < T_K\}} \sup_{t \in [0,T]} |X_t - \tilde X_t|^2 
            \leq \\
            \le C \left[ T \int_0^T \indiq_{\{s < T_K\} } | b(X_s, \mu_s) - b(\tilde X_s, \tilde \mu_s) |^2 ds  + T M^2_K \int_0^T \indiq_{\{s < T_K\}} | \mu^{1/\alpha}_s (f) - \tilde \mu^{1/\alpha}_s(f) |^2 ds \right.\\
\left.  + \sup_{t\in [0,T]} \left|\int_{[0,t]\times B^\ast_K} \indiq_{ \{ s \leq T_K \} } [\mu^{1/\alpha}_{s-}(f) - \tilde\mu^{1/\alpha}_{s-}(f) ] z \tilde M(ds,dz) \right|^2   \right] ,
  \end{multline*}
where 
\begin{equation}\label{eq:def:C_K}
M_K \coloneqq \int_{B^c_K} z \nu^\alpha(dz).
\end{equation}

Using the Burkholder-Davis-Gundy inequality to deal with the stochastic integral term, we obtain
\begin{multline*}
        \E\left[\indiq_{\{T<T_K\}}  \sup_{t\in [0,T]}|X_t - \tilde X_t|^2 \right]  \leq C_{K,T} \left[ \int_0^T  \E\left[\indiq_{\{s< T_K\}}| b(X_s,\mu_s) - b(\tilde X_s, \tilde \mu_s)|^2\right] ds \right.\\ 
         \left. + \int_0^T \E\left[\indiq_{\{s< T_K\}}| \mu^{1/\alpha}_s(f) - \tilde\mu^{1/\alpha}_s(f)|^2 \right] ds \right] .    
\end{multline*}
Using the first inequality of \eqref{eq:controlmuf} with $p=2,$ Jensen's inequality and the Lipschitz continuity of $f,$ we obtain 
\begin{multline}\label{eq:similar}
  | (\mu_s(f))^{1/\alpha} - (\tilde\mu_s(f))^{1/\alpha}|^2 \leq C 
            \left| \mu_s(f) - \tilde\mu_s(f)\right|^2 = C\left( \E[f(X_s)  \, | \, S^\alpha] - \E[f(\tilde X_s)  \, | \, S^\alpha]\right)^2\\
   \leq C \E[ |f(X_s) - f(\tilde X_s)|^2 \, | \, S^\alpha ]  \leq C \E[  |X_s - \tilde X_s |^2 \, | \, S^\alpha ].
    \end{multline} 
We conclude that $ \indiq_{\{s< T_K\}}   | (\mu_s(f))^{1/\alpha} - (\tilde\mu_s(f))^{1/\alpha}|^2 \le    C \E[\indiq_{\{s< T_K\} }   |X_s - \tilde X_s |^2 \, | \, S^\alpha ], $ since $ \{s< T_K\}   $ is $ S^\alpha-$measurable.

Using this and Assumption \ref{ass:b}\ref{itm:bLip} together with the fact that $W_1(\mu_s,\tilde\mu_s) \le W_2(\mu_s,\tilde\mu_s),$ 
\begin{equation*}
    \begin{split}
        \E\left[\indiq_{\{T<T_K\}}\sup_{t\in [0,T]}|X_t - \tilde X_t|^2 \right] & \leq C_{K,T} \left[ \int_0^T  \E\left[\indiq_{\{s< T_K\}}\left(|X_s - \tilde X_s|^2 + W^2_2(\mu_s,\tilde\mu_s)\right)\right] ds\right.\\
        & \left. + \int_0^T \E\left[\E\left[\indiq_{ \{s < T_K\} } |X_s - \tilde X_s|^2 \, | \, S^\alpha\right]\right] ds
        \right] . 
    \end{split}
\end{equation*} 
By definition of the Wasserstein-$2$ distance, $W^2_2(\mu_s, \tilde\mu_s) \leq \E\left[|X_s - \tilde{X}_s|^2\, | \, S^\alpha\right] $, such that 
\begin{equation}\label{eq:uniqueness:gronwall}
    \begin{split}
        \E\left[\indiq_{\{T<T_K\}}\sup_{t\in [0,T]}|X_t - \tilde X_t|^2 \right] & \leq  C_{K,T} \int_0^T \E\left[\indiq_{ \{ s < T_K \} } \sup_{r \in [0,s]} |X_r - \tilde X_r|^2 \right] ds .
    \end{split}
\end{equation}
Notice that due to our a priori bound \eqref{eq:mufin}, $\E\left[\indiq_{ \{ s < T_K \} } \sup_{r \in [0,s]} |X_r - \tilde X_r|^2 \right]  $ is finite, such that the above inequality implies, by  Gronwall's inequality, that  
\begin{equation}
    \begin{split}
        \E\left[\indiq_{\{T < T_K\}}\sup_{t\in [0,T]}|X_t - \tilde X_t|^2 \right] = 0.
    \end{split}
\end{equation} 
Since  $\lim_{ K \to \infty} \indiq_{\{T < T_K\}} = 1 $ almost surely, the assertion follows by  monotone convergence.

\subsubsection{The case $ \alpha < 1$} 
We now discuss the case $ \alpha < 1.$ In this case $S^\alpha$ is of bounded variation such that we may use the $L^1$-norm instead of the $L^2$-norm. We have
\begin{multline*}
           \sup_{t \in [0,T]}\indiq_{\{t < T_K\}} |X_t - \tilde X_t| 
            \leq 
             \int_0^T \indiq_{\{s < T_K\} } | b(X_s, \mu_s) - b(\tilde X_s, \tilde \mu_s) | ds   \\
             +   \int_{[0,T ]\times\R_+ } \indiq_{ \{ s \leq T_K \} }   | \psi ( X _{s-} , \mu_{s-})  \indiq_{ \{ z \le  f (  X_{s-}) \}} -  \psi (\tilde  X _{s-} , \tilde \mu_{s-})  \indiq_{ \{ z \le  f ( \tilde X_{s-}) \}} | \bar \pi (ds,dz)\\
                +  \int_{[0,T ]\times \R_*} \indiq_{ \{ s \leq T_K \} } | \mu^{1/\alpha}_{s-}(f) - \tilde\mu^{1/\alpha}_{s-}(f) |  |z|  M(ds,dz)  .     
  \end{multline*}
Using Assumption \ref{ass:b}\ref{itm:bLip}, \eqref{eq:controlmuf} with $p=1$ and similar arguments as in \eqref{eq:similar}-\eqref{eq:uniqueness:gronwall}, now with the $L^1$-norm instead of the $L^2$-norm, we have that 
\begin{multline*}
 \E \left[ \int_0^T \indiq_{\{s < T_K\} } | b(X_s, \mu_s) - b(\tilde X_s, \tilde \mu_s) | ds   \right.\\
 \left. + \int_{[0,T ]\times \R_*} \indiq_{ \{ s \leq T_K \} } | \mu^{1/\alpha}_{s-}(f) - \tilde\mu^{1/\alpha}_{s-}(f) |  |z|  M(ds,dz)  \right] \\
 \le C_{K, T }  \E\left[\indiq_{\{T < T_K\}}\sup_{t\in [0,T]}|X_t - \tilde X_t| \right] . 
\end{multline*}
Finally we use that both $f $ and $ \psi $ are bounded and get 
\begin{multline*}
\E \left[  \int_{[0,T ]\times\R_+ } \indiq_{ \{ s \leq T_K \} }   | \psi ( X _{s-} , \mu_{s-})  \indiq_{ \{ z \le  f (  X_{s-}) \}} - \psi (\tilde  X _{s-} , \tilde \mu_{s-})  \indiq_{ \{ z \le  f ( \tilde X_{s-}) \}} | \bar \pi (ds,dz) \right] \\
\le   \E \int_0^T \indiq_{ \{ s \leq T_K \} } \left[ \| \psi\|_\infty |   f (  X_{s-}) -   f ( \tilde X_{s-}) |  + \|f\|_\infty  |  \psi ( X _{s-} , \mu_{s-})-  \psi (\tilde  X _{s-} , \tilde \mu_{s-}) | \right] ds.
\end{multline*}

We then use their Lipschitz continuity to conclude the proof as before using Gronwall's lemma.  

The existence of a strong solution of \eqref{eq:limitsystem} follows from a Picard iteration. The proof is postponed to the Appendix section \ref{app:strong_ex_lim}. $ \qed $

\begin{remark}\label{rem:bigjumps}
It is difficult to include \textit{main jumps} in the case $ \alpha > 1 .$ This is due to the fact that the natural way of controlling them is by using the $L^1$-norm, while the small jumps of the stable integral need to be controlled in $L^2 .$ Such a difficulty has already been remarked in the paper by Carl Graham, \cite{graham92}. It is possible to deal with main jumps in the presence of martingale terms which are not of bounded variation using other techniques, see e.g. \cite{ELL1} and \cite{ELL2}, where we applied a space transform and worked in $L^1, $ using the technique of \cite{graham92} who uses  the Burkholder-Davies-Gundy inequality in $L^1 $ to deal with the martingale term. In our present framework, the martingale is discontinuous such that this approach does not seem to be feasible. Therefore we have decided to disregard main jumps in the case $ \alpha > 1.$ 
\end{remark}

\section{Representing the interaction term of the finite particle system as a stochastic integral against a stable process}\label{sec:coupling}
\subsection{Main representation result}
To prove Theorem \ref{thm:prop:chaos}, we cut time into time slots of length $ \delta > 0.$ We will choose $ \delta = \delta (N )$ such that, as $N \to \infty, $ $\delta (N) \to 0$ and (at least) $ N \delta (N) \to \infty.$  The precise choice will be given in Section \ref{sec:delta_choice} below. 
 We write $\tau_s=k\delta$ for $k\delta < s \leq (k+1)\delta,\,  k\in\N, s > 0.$  A first step of the proof is a representation of  the interaction term 
\begin{equation}\label{eq:atn}
A_{t}^{N}\coloneqq \frac 1{N^{1/\alpha}}\sum_{j=  1}^N\int_{[0,t]\times\R_+\times\R }u \indiq_{ \{ z \le  f ( X^{N, j}_{s-}) \}} \pi^j (ds,dz,du)
\end{equation}
in terms of a stochastic integral against a stable process. Recall that  $r_t(N,\delta)$ has been defined in Theorem \ref{thm:prop:chaos}.  
\begin{theorem}\label{prop:representation_finite_syst}
Grant Assumptions \ref{ass:b}--\ref{ass:nu0+}. 
For any $N\in \N^\ast$ and $\delta \in (0,1)$ such that $ 2 \delta \|f\|_\infty < 1, $ there exists, on an extension of $(\Omega,\mathcal{A},\mathbf{P})$ which depends on $ N $ and $ \delta,$ a strictly $\alpha$-stable L{\'e}vy process $S^{N,\alpha}$, independent of the Poisson random measures $(\bar \pi^i)_{i=1}^N$ and of the initial conditions  $(X^{N,i}_0)_{i=1}^N$, such that 
\begin{equation}\label{eq:atnapprox}
A_t^N =  \int_{[0,t]} \left(\mu^{N}_{\tau_s}(f)\right)^{1/\alpha} d S^{N,\alpha}_s +  R^{N}_t , 
\end{equation}
where $ R_t^N $ is an error term for which we have the following control.
\begin{description}
\item {i)} 
If $1<\alpha<2$, then for $ \delta = \delta (N) $ sufficiently small,
\begin{equation*}
\E[ | R_t^{N } | ]  \leq C_{t} r_t(N, \delta).
\end{equation*}
\item{ii)}
If $ 0<\alpha < 1$, then for $ \delta = \delta (N) $ sufficiently small,  
\begin{equation*}
\E[  \norme{R_t^N}_{d_{\am}} ]  \leq  C_{t} r_t(N, \delta) .
\end{equation*}
\end{description}
In both cases, $C_{t}$ is a constant which is non-decreasing in $t$. 
\end{theorem}

The remainder of this section is devoted to the proof of Theorem \ref{prop:representation_finite_syst} which is the main tool to prove Theorem \ref{thm:prop:chaos}. The proof is given in Section \ref{subs:Y} and it uses  the following steps.
 
\textbf{Step 1.} We replace $A_t^N $ by its 
%time discretized
 version with time-discretized intensity
\begin{equation*}\label{eq:adtn}
A_t^{N,\delta}\coloneqq \frac 1{N^{1/\alpha}}\sum_{j =  1}^N\int_{[0,t]\times\R_+\times\R }u \indiq_{ \{ z \le  f ( X^{N, j}_{\tau_{s}}) \}} \pi^j (ds,dz,du) .
\end{equation*} 
This is done directly in the proof of Theorem \ref{prop:representation_finite_syst}, Section \ref{subs:Y}. The error made due to this time discretization is controlled thanks to a general error bound stated in Proposition \ref{prop:discretis0} below. 
Notice that this time discretization does only apply to the interaction term; we do not apply it to the whole process. 

\textbf{Step 2.} We show that any increment of $ A^{N,\delta } $ can be represented as the product of a conditional Poisson random variable (the total number of jumps per time interval) and the
increment of a stable process. This is the content of Proposition \ref{lem:preliminary_results}.

\textbf{Step 3.} We replace, in the proof of Theorem \ref{prop:representation_finite_syst}, Section \ref{subs:Y}, the suitably renormalized Poisson random variable by its expectation to conclude our proof.

\subsection {Representation of the discretized  increment of the interaction term  }
Let $0\leq s < t$ and define 
\begin{equation}\label{eq:def:A^N_s_t}
    \begin{split}
        A^{N}_{s,t} \coloneqq \frac{1}{N^{1/\alpha}}\sum_{j=1}^N \int_{]s, t]\times \R_+ \times \R} u \indiq_{\{z\leq f(X^{N,j}_{s})\}} \pi^j(dr,dz,du)
    \end{split}
\end{equation}
and
\begin{equation}\label{eq:def:P^N_s_t}
    \begin{split}
        P^{N}_{s,t} \coloneqq \sum_{j=1}^N \int_{]s, t]\times \R_+ \times \R}  \indiq_{\{z\leq f(X^{N,j}_{s})\}} \pi^j(dr,dz,du).
    \end{split}
\end{equation}

The following proposition combines a (conditional version of) Proposition \ref{prop:scaling} with a quantified version of the stable central limit theorem. Recall that we suppose that $ \nu $ is heavy-tailed according to Definition \ref{def:strongdomain} above. The parameter $ \gamma $ is the one of that definition. 

\begin{proposition}\label{lem:preliminary_results}
For all $N\in \N^\ast$ and all $0\leq s < t$, such that $ 2 (t-s) \|f\|_\infty < 1, $ 
there exists, on an extension %$(\tilde \Omega,  \mathcal{\tilde A}, \mathbf{\tilde P})$ 
of the original probability space  $(\Omega,\mathcal{A},\mathbf{P}),$ a strictly $\alpha$-stable random variable $S^{N,\alpha}_{s,t}$ which is independent of $P^N_{s,t},$ of $\F_s$ and of $( \bar \pi^i)_{i \geq 1},$  such that,
\begin{description}
\item{i)} if $1<\alpha<2$, 
\begin{equation}\label{eq:representation_L1_bound}
\E \left[ \left| A^N_{s,t} - \left( \frac{P^N_{s,t}}{N}\right)^{1/\alpha} S^{N,\alpha}_{s,t} \right| \right] \leq   C (t-s)^{\frac{1}{\alpha}}  g(N (t-s)), 
\end{equation}
where 
\begin{equation}\label{eq:def:g}
g(x) = \begin{cases}
x^{-\frac{\gamma}{\alpha}} & 0 < \gamma  < 2-\alpha \\
x^{-\frac{2-\alpha}{\alpha}} & \gamma > 2-\alpha , 
\end{cases}
\end{equation}

\item{ii)} if $\alpha<1$, for any $\am < \alpha$,
\begin{equation}\label{eq:representation_L1_bound:alpha<1}
\E \left[  d_{\am}\left(A^N_{s,t}, \left( \frac{P^N_{s,t}}{N}\right)^{1/\alpha} S^{N,\alpha}_{s,t}\right) 
\right] \leq C [ (t-s)^{\frac{\am}{\alpha}} g(N (t-s)) +   e^{- C N |t-s|} ] ,
\end{equation}
where
\begin{equation}\label{eq:def:g:alpha<1}
g(x) = 
x^{-1} + x^{-\frac{\gamma}{\alpha}} + x^{\frac{\alpha-1}{\alpha}}, \,  \gamma \neq 1-\alpha .
\end{equation}
\end{description}

Moreover, for any $N \in \N^\ast$ and $\delta\in (0,1)$ such that $2\delta\norme{f}_\infty < 1$, there exists, on an extension of the original probability space, an i.i.d. sequence $(S^{N,\alpha}_{k\delta, (k+1)\delta})_{k\geq 0}$ of strictly stable random variables such that, for all $k\geq 0$, $S^{N,\alpha}_{k\delta, (k+1)\delta}$ is independent of $\F_{k\delta}$ and of $(\bar\pi^{i})_{i\geq 1}$ and satisfies the bound \eqref{eq:representation_L1_bound} (if $\alpha>1$) and \eqref{eq:representation_L1_bound:alpha<1}  (if $\alpha<1$) with $s=k\delta$ and $t=(k+1)\delta$. 
\end{proposition}

\begin{proof}
Using basic properties of Poisson random measures, there exists an i.i.d. sequence of random variables $U^{(l)}_{s,t}, l \geq 1, $ distributed as $ \nu, $ independent of $P^N_{s,t},$ of $\F_s$ and of $ \bar \pi^i , i \geq 1, $ such that almost surely, 
\[
  A^{N}_{s, t} =\frac{1}{N^{1/\alpha}}  \tilde S_{P^{N}_{s, t}} , \mbox{ where we put } \tilde S_n \coloneqq \sum_{l=1}^n  U^{(l)}_{s,t}   ,
\]
which we rewrite as 
\begin{equation}\label{eq:14}
A^{N}_{s, t} =  \left(\frac{P^{N}_{s, t}}{N}\right)^{1/\alpha} \frac{1}{(P^{N}_{s, t})^{1/\alpha}} \tilde S_{P^{N}_{s, t}}.
\end{equation}

Since $\nu$ belongs to the domain of attraction of an $\alpha$-stable law,  the weak limit, as $n\to +\infty$, of the sequence $\left(\frac{1}{n^{1/\alpha}}\tilde{S}_n\right)_{n\geq 1},$ is the law of an $\alpha-$stable random variable $ S^\alpha.$

\sloppy Let $\mu_n$ be the optimal coupling minimizing $W_1\left(\frac{1}{n^{1/\alpha}}\tilde{S}_n, S^\alpha\right)$ if $\alpha>1$ and $W_{d_{\am}}\left(\frac{1}{n^{1/\alpha}}\tilde{S}_n, S^\alpha\right)$  for $\alpha<1,$ for any fixed $ n \geq 1$.  
These optimal couplings exist by Theorem 4.1 in \cite{villani}. Denote by $\mu^{(1)}_n$ their first marginal in both cases.   
Then a general coupling lemma, stated e.g. in Lemma 3.12 of the arXiv version  \cite{prodhommeArxiv} of  \cite{Prodhomme}, see also Section \ref{app:coupling} below, implies that there exists a measurable function $G_n : \R \times (0,1) \to \R$ such that $ \left( X_n, G_n\left(X_n, V\right)\right) \sim \mu_n $ whenever $(X_n, V) \sim \mu^{(1)}_n \otimes U(0,1)$. In particular, this holds for $X_n = \frac{1}{n^{1/\alpha}}\tilde S_n$ and $V \sim U(0,1)$ independent of it.   
Since the second marginal of $\mu_n$ is $\mathcal{L}(S^\alpha)$, this means that 
\begin{equation}\label{eq:G_n=S^alpha}
G_n\left(\frac{1}{n^{1/\alpha}}\tilde S_n, V\right) \sim S^\alpha .
\end{equation}
We also let for $n=0$, $ G_0 ( x, V ) $ be a random variable that does not depend on $x $ and which is distributed as $ G_0(x, V)  \sim S^\alpha.$ 
Let us finally introduce $ G : \N \times \R \times (0, 1 ) \to \R $ by putting 
\[ G (n, x, v ) \coloneqq \sum_{l=0}^\infty 1_{\{ n = l \} } G_l ( x, v ) .\]   

We now define a first extension of our original probability space by adding to it a uniform random variable $V\sim U(0,1),$ independent of the Poisson random measures $(\pi^i)_i$ and of the random variables $ (\tilde S_n)_n. $ \footnote{Note that this first extension does not depend on $ s, t  ,$ nor on $ N , \delta.$}  
Then, for any fixed $N\in \N^*$ and $0\leq s < t$, we define 
\begin{equation}\label{eq:def:S^Nalpha_s_t}
S^{N,\alpha}_{s,t} \coloneqq G \left( P^{N}_{s, t},  \frac{1}{(P^{N}_{s, t})^{1/\alpha}} \tilde S_{P^{N}_{s, t}} , V \right) . 
\end{equation}
We now prove that $S^{N,\alpha}_{s,t}$ is an $\alpha$-stable variable, which is independent of  $P^N_{s,t},$ of $\F_s$ and of $ \bar \pi^i , i \geq 1. $   
Let $ A \in \sigma \{ P^N_{s,t}\} \vee \sigma \{  \bar \pi^i , i \geq 1\} \vee \F_s . $ We have 
\begin{eqnarray*}
\E[\phi(S^{N,\alpha}_{s,t}) \indiq_A ] & =& \E \left[ \phi \left( G \left( P^{N}_{s, t},  \frac{1}{(P^{N}_{s, t})^{1/\alpha}} \tilde S_{P^{N}_{s, t}} , V \right)\right) \indiq_A \right] \nonumber \\
& = &\sum_{n=0}^\infty \E\left[\phi\left(G_n\left( \frac{1}{n^{1/\alpha} } \tilde S_n , V \right) \right) \indiq_{\{P^{N}_{s,t} = n \}}  \indiq_A \right] \nonumber \\
& = &\sum_{n=0}^\infty \E\left[\phi\left(G_n\left( \frac{1}{n^{1/\alpha} } \tilde S_n , V \right) \right)\right] \mathbf{P}(\{P^{N}_{s,t} = n \} \cap  A ) \nonumber \\
&=&  \sum_{n=0}^\infty \E\left[\phi\left(S^\alpha\right)\right] \mathbf{P}(\{ P^{N}_{s, t} = n\} \cap A ) 
= \E\left[\phi\left(S^\alpha\right)\right] \mathbf{P} ( A )  ,
\end{eqnarray*}
for any measurable bounded test function $\phi$. Here, we used that $ ((\tilde S_n)_n, V)  $ is independent of $P^N_{s,t},$ of $\F_s$ and of $ \bar \pi^i , i \geq 1, $ to obtain the third equality. Thus, $S^{N,\alpha}_{s,t}$ is indeed an $\alpha$-stable variable, which is independent of  $P^N_{s,t},$ of $\F_s$ and of $ \bar \pi^i , i \geq 1.$ 

 Finally, the bounds in \eqref{eq:representation_L1_bound} and \eqref{eq:representation_L1_bound:alpha<1} follow using the representation \eqref{eq:14} for $A^N_{s,t}$.

\noindent{\textbf{Case} $\pmb{\alpha>1}.$}
Employing that, conditionally on $\{P^N_{s,t}=n\}$, $G\left(P^N_{s,t}, \frac{\tilde{S}_{P^N_{s,t}}}{(P^N_{s,t})^{1/\alpha}}, V\right) = G_n\left(\frac{1}{n^{1/\alpha}}\tilde S_n, V \right) $, that $\mathcal{L}\left(\frac{\tilde S_n}{n^{1/\alpha}},G_n\left(\frac{1}{n^{1/\alpha}}\tilde S_n, V\right) \right)$ is the optimal coupling $\mu_n$ for the $W_1$ distance and  \eqref{eq:G_n=S^alpha},
 \begin{multline}\label{eq:proof:representation_L1_bound}
         \E \left[ \left| A^N_{s,t} - \left( \frac{P^N_{s,t}}{N}\right)^{1/\alpha} S^{N,\alpha}_{s,t} \right| \right]  = \sum_{n > 0} \left(\frac{n}{N}\right)^{1/\alpha} \E\left[\left| \frac{\tilde S_n}{n^{1/\alpha}}- G_n\left(\frac{\tilde S_n}{n^{1/\alpha}}, V\right)\right|\right] \mathbf{P}(P^{N}_{s, t}=n)\\
             = \frac{1}{N^{1/\alpha}}\sum_{n >  0} n^{1/\alpha} W_1\left( \frac{\tilde S_n}{n^{1/\alpha}}, G_n\left(\frac{\tilde S_n}{n^{1/\alpha}}, V\right) \right) \mathbf{P}(P^{N}_{s, t}=n) \\
             = \frac{1}{N^{1/\alpha}} \sum_{n > 0} n^{1/\alpha} W_1\left( \frac{\tilde S_n}{n^{1/\alpha}}, S^\alpha \right) \mathbf{P}(P^{N}_{s, t}=n) .
    \end{multline}     
Based on Example 2 in \cite{xu}, for any $n\in \N^*$, $W_1\left(\frac{\tilde S_n}{n^{1/\alpha}},S^\alpha\right) \leq C g(n)$, for a positive constant $C,$ where the function $g$ is given in \eqref{eq:def:g} 
and satisfies 
that  $x \mapsto x^{1/\alpha}g(x)$ is concave on $ \R_+$ for any choice of $\gamma > 0 .$ As a consequence,  we can apply Jensen's inequality in  \eqref{eq:proof:representation_L1_bound} to obtain
\begin{multline*}
 \frac{1}{N^{1/\alpha}} \sum_{n > 0} n^{1/\alpha} W_1\left( \frac{\tilde S_n}{n^{1/\alpha}}, S^\alpha \right) \mathbf{P}(P^{N}_{s, t}=n) \leq \frac{C}{N^{1/\alpha}}\sum_{n > 0} n^{1/\alpha} g(n)  \mathbf{P}(P^{N}_{s, t}=n) \\
= \frac{C}{N^{1/\alpha}} \E[(P^N_{s,t})^{1/\alpha} g(P^N_{s,t})] \leq \frac{C}{N^{1/\alpha}} \E[P^{N}_{s,t}]^{1/\alpha} g(\E[P^N_{s,t}])
 \leq C \|f\|_\infty^{1/ \alpha}  |t-s|^{1/\alpha} g(\underline f N |t-s|) ,
\end{multline*}
where we have employed that, conditionnally on $\F_{s}$, $ P^{N}_{s, t} \sim Pois((t-s) N \mu^N_{s}(f))$ and Assumption \ref{ass:f}\ref{itm:fbdd} to upperbound $\E[P^{N}_{s,t}]$ and \ref{ass:f}\ref{itm:flbdd} to upperbound $ g(\E[P^N_{s,t}])$ (using that $g$ is decreasing). 

\noindent \textbf{Case} $\pmb{\alpha<1}.$
In this case, we introduce $G \coloneqq \left\{\frac{1}{2} |t-s| N \underline{f}  \le P^{N}_{s, t} \leq 2  |t-s| N \|f\|_\infty  \right\}$.
It is easy to show, using Chernoff bounds for Poisson random variables, that $\mathbf{P}(G^c\, | \, \mathcal{F}_{s}) 
\le c e^{-C N |t-s|},$ for some constants $c, C > 0 $ not depending on $N.$ Then we start from 
\begin{eqnarray*}
& & \E\left[ d_{\am}\left(A^N_{s, t}, \left(\frac{P^{N}_{s, t }}{N}\right)^{1/\alpha} S^{N,\alpha}_{s, t}\right)\right] \\
& & \leq \E\left[ \left( \left(\frac{P^{N}_{s, t }}{N}\right)^{1/\alpha} \vee \left(\frac{P^{N}_{s, t }}{N}\right)^{\am/\alpha} \right) 
 d_{\am}\left(
\frac{1}{(P^{N}_{s, t })^{1/\alpha}} \tilde S_{P^{N}_{s, t }},  S^{N,\alpha}_{s, t} \right)  \indiq_G \right] \\
 &&\quad \quad + \E\left[  \left|A^N_{s, t} - \left(\frac{P^{N}_{s, t }}{N}\right)^{1/\alpha} S^{N,\alpha}_{s, t} \right|^{\am}  \indiq_{G^c} \right] \eqqcolon T_1 + T_2.
\end{eqnarray*}
To control $ T_2,$ we use the sub-additivity of $ x \mapsto x^\am $ and H\"older's inequality with $ q ' /\am $ where $ q' \in ] \am, \alpha [ $ and associated conjugate exponent $ p'$ such that $ 1/ p' + \am/ q' = 1, $ to upper bound 
\begin{eqnarray*}
T_2 &\le& \left(\E [  |A^N_{s, t}|^{q'} ]\right)^{ \am/ q'} \mathbf{P} ( G^c )^{1/p'} + \left(\E [ | S^{N,\alpha}_{s, t} |^{q'} ]\right)^{\am/ q'} \left(\E \left[ \left(\frac{P^{N}_{s, t }}{N}\right)^{{\am p'}/\alpha} \indiq_{G^c} \right] \right)^{1/ p'} 
\\
&\le & \left( C N^{ 1 - q'/\alpha} |t-s| \right)^{\am/q'}  \mathbf{P} ( G^c )^{1/p'}  + C \left(\E \left[ \left(\frac{P^{N}_{s, t }}{N}\right)^{{\am p'}/\alpha} \indiq_{G^c} \right]\right)^{1/ p'}  \\
&\le & c e^{-\tilde C N |t-s|},
\end{eqnarray*}
where $C$ is an upper bound on $(\E [ | S^{N,\alpha}_{s, t} |^{q'} ])^{\am/ q'}$ and where we have used H\"older's inequality once more to upper bound $\E ( \left({P^{N}_{s, t }}/{N}\right)^{{\am p'}/\alpha} \indiq_{G^c} ).$ 
To control $T_1,$ notice that we have $ \frac{P^{N}_{s, t }}{N} \le 1 $ on $ G$ (recall that we assumed that $ 2 \|f\|_\infty (t-s) \le 1 $), such that 
\[  \left(\frac{P^{N}_{s, t }}{N}\right)^{1/\alpha} \vee \left(\frac{P^{N}_{s, t }}{N}\right)^{\am/\alpha} = \left(\frac{P^{N}_{s, t }}{N}\right)^{\am/\alpha}.\]
Therefore, using analogous arguments as in the case $\alpha>1$,  
\begin{eqnarray*}
& & \E\left[ 
 \left( \left(\frac{P^{N}_{s, t }}{N}\right)^{1/\alpha} \vee \left(\frac{P^{N}_{s, t }}{N}\right)^{\am/\alpha} \right) 
 d_{\am}\left(\frac{1}{(P^{N}_{s, t })^{1/\alpha}} \tilde S_{P^{N}_{s, t }},  S^{N,\alpha}_{s, t} \right) \indiq_G
\right] \\
& & = \E\left[ 
  \left(\frac{P^{N}_{s, t }}{N}\right)^{\am/\alpha} 
 d_{\am}\left(\frac{1}{(P^{N}_{s, t })^{1/\alpha}} \tilde S_{P^{N}_{s, t }} ,  S^{N,\alpha}_{s, t} \right) \indiq_G
\right] \\
& & \leq \frac{1}{N^{\am/\alpha}}\sum_{n\geq 0} n^{\am/\alpha} \E\left[  d_{\am}\left(\frac{\tilde S_{n}}{n^{1/\alpha}}  ,  G_n\left(\frac{\tilde S_{n}}{n^{1/\alpha}}, V\right) \right) \right] \mathbf{P}( P^{N}_{s, t } = n) \\
& & = \frac{1}{N^{\am/\alpha}} \sum_{n\geq 0} n^{\am/\alpha} d_{W_\am}\left(\frac{\tilde S_n}{n^{1/\alpha}}, S^\alpha \right) \mathbf{P}( P^{N}_{s, t } = n) .
\end{eqnarray*}
Based on Example 2 in \cite{chen22}, for any $n\in \N^*,$ 
$d_{W_\am}\left(\frac{\tilde S_n}{n^{1/\alpha}}, S^\alpha \right) \leq C g(n)$, where $g$ is given in
 \eqref{eq:def:g:alpha<1}
and satisfies that  $x\mapsto x^{\am/\alpha} g(x)$ is concave on $\R_+$ for any $\gamma>0$. Then, as before, we apply Jensen's inequality to conclude the proof of \eqref{eq:representation_L1_bound:alpha<1}: 
\begin{eqnarray*}
&&\frac{1}{N^{q/\alpha}} \sum_{n\geq 0} n^{\am/\alpha} d_{W_\am}\left(\frac{\tilde S_n}{n^{1/\alpha}}, S^\alpha_1 \right) \mathbf{P}( P^{N}_{s, t } = n)
\leq    C |t-s|^{\am/\alpha} g(\underline{f} N |t-s|) .
\end{eqnarray*}

Last, if we consider, on a further extension of $(\Omega,\mathcal{A},\mathbf{P}) $, an i.i.d. sequence $(V_k)_{k\geq 0}$ of uniform random variables, independent of anything else, and we put (recalling \eqref{eq:def:S^Nalpha_s_t}) 
%The reasoning above shows in particular that, if we consider an i.i.d. sequence $(V_k)_{k\geq 0}$ of uniform random variables independent of anything else and we put, 
\[ S^{N, \alpha}_{ k \delta, (k+1) \delta } \coloneqq G \left( P^N_{ k \delta, (k+1) \delta} ,  \frac{1}{(P^{N}_{k \delta , (k+1) \delta})^{1/\alpha}} \tilde S_{P^{N}_{k \delta , (k+1) \delta}} , V_k  \right) ,
\]
$k\geq 0$, then the previous reasoning yields that $(S^{N, \alpha}_{ k \delta, (k+1) \delta } )_{k\geq 0}$ is a sequence of i.i.d. strictly stable random variables, and that for each $ k \geq 0, $ $S^{N,\alpha}_{k\delta,(k+1)\delta} $ is independent of $\mathcal{F}_{k\delta}$ and of $(\bar\pi^{i})_{i\geq 1}$.   
This observation concludes the proof. 
\end{proof}

\subsection{Construction of an $\alpha$-stable L{\'e}vy process}\label{subs:bridge}

Let  $(S^{N,\alpha}_{k\delta,(k+1)\delta} )_{k\geq 0}$ the i.i.d. family of strictly stable variables obtained in Proposition \ref{lem:preliminary_results}. 
Then we have the following proposition. 
\begin{proposition}\label{lem:construction_limit_LP}
For each $N \in \N^\ast$ and $ \delta \in ( 0,1) $ such that $ 2 \|f\|_\infty \delta < 1$, there exists, on an extension of $(\Omega, \mathcal{A},\mathbf{P})$ depending on $N$ and $\delta$, a strictly $\alpha$-stable process $S^{N,\alpha}$ independent of $ \F_0 $ and of $ \bar \pi^i, i \geq 1, $  such that almost surely, 
\begin{equation}\label{eq:as}
S^{N, \alpha}_{(k+1)\delta}- S^{N, \alpha}_{k\delta} = \delta^{1/\alpha} S^{N,\alpha}_{k\delta, (k+1) \delta}
\end{equation}
for all $k\geq 0. $ 
\end{proposition}
\begin{proof}
Fix any $\delta>0$ and consider the stable process $ ( S^\alpha_t )_{ t \geq 0 }, $ starting from $ S_0^\alpha = 0,$ defined on some probability space. 
Then, by Theorem 8.5 in \cite{kallenberg}, the joint law of $ ( S_\delta^\alpha, ( S^\alpha_t )_{ t \in [ 0, \delta ] }) $ can be disintegrated into the product of $\mathcal{L}(S^\alpha_\delta)$ and a probability kernel $Q$ from $\R$ to the Skorokhod space $\Dd$ such that $Q$ is $\mathcal{L}(S_\delta^\alpha)$-a.s. unique and satisfies 
\[ Q(S_\delta^\alpha, \cdot) = \mathcal{L}(( S^\alpha_t )_{ t \in [ 0, \delta ]}  | S_\delta^\alpha)(\cdot) \text{ a.s.}\]  
Notice that $Q$ depends on $\delta.$

Next, consider the product spaces $\R^\N$ and $\Omega' \coloneqq \prod_{k\geq 0} \Dd$, which also depends on $ \delta,$ endowed with the product $\sigma$-algebras, and define the product kernel 
\[ Q^\infty( (x^{(k)})_{k\geq 0} , d\gamma) \coloneqq \otimes_{k\geq 1} Q (x^{(k)}, d\gamma^{(k)}) \, . \]
Then $Q^\infty( (x^{(k)})_{k\geq 0} , d\gamma) $ is a probability kernel, from $\R^\N$ to $\Omega', $ thanks to the extension theorem of Ionescu-Tulcea (see e.g. Theorem 8.23 and Corollary 8.25 in \cite{kallenberg}). 
Define now  
\[ \phi \, :\, \Omega' \to D(\R_+, \R), \; 
        (\gamma^{(k)})_{k\geq 0}  \mapsto (\phi_t)_{t\geq 0}
\]
by $\phi_0 = 0$ and for any $ t > 0,$
\begin{equation*}
  \phi_t \coloneqq \sum_{k\geq 0} \indiq_{] k \delta, (k+1)\delta]}(t) \left(\sum_{l=0}^{k-1} \gamma^{(l)} ( \delta )  + \gamma^{(k)}(t-k\delta) \right), 
\end{equation*}
with the convention $\sum_{l=0}^{-1}  = 0$. 

Introduce finally
\[ X^{(k)} = \delta^{1/\alpha} S^{N,\alpha}_{k\delta, (k+1) \delta} \; ,\]
defined thanks to Proposition \ref{lem:preliminary_results} for any $ k \geq 0 $ on an extension of the original probability space. Then by construction, $(\phi_t)_{t \geq 0  } $ is an $\alpha$-stable process under $\mathbf{P}  \otimes Q^\infty( (X^{(k)})_{k\geq 1}, \cdot) .$ 
Setting $(S_t^{N,\alpha})_t \coloneqq (\phi_t)_t $ concludes the proof. 

\end{proof}

\subsection{Errors due to time discretization}\label{sec:discretis}
We now state a generic result expressing the error that is due to time discretization. 

\begin{proposition}\label{prop:discretis0}
 Grant Assumptions \ref{ass:b}\ref{itm:bbdd}, \ref{ass:f}\ref{itm:fbdd}, \ref{ass:psi}\ref{itm:psibdd}, \ref{ass:nu0} and \ref{ass:f+}.  
Then for any  $ 1 \le i \le N, $ for any $ s, t \geq 0 $ such that $ |s-t | \le 1,$ and for $  \am < \alpha,$ 
($\am \geq 1, $ if $ \alpha > 1$), we have \\
(i) 
\begin{equation*}
\E[|  f(X^{N, i }_t)  - f(X^{N, i }_{s }) |]\le  C \begin{cases}   N^{1/\am-1 /\alpha} |t-s|^{1/\am} , & \mbox{ if }  \alpha \in (1,2) \\ 
N^{ 1 - \am/ \alpha } |t-s| , & \mbox{ if } \alpha \in (0,1) .
\end{cases}
 \end{equation*}
(ii) Moreover, if  $ \alpha > \am  > 1,$ then we also have
\[
\E [ |X^{N, i }_t  - X^{N, i }_s |^\am ] \le C N^{ 1 - \am/ \alpha} |t-s| .
\]
\end{proposition}
\begin{proof}
We start with the proof of item (i) in case $ \alpha < 1 .$ By exchangeability, it suffices to consider $i = 1.$ Suppose w.l.o.g. that $s \le t .$ By Ito's formula, 
\begin{multline*}
f(X^{N, 1 }_{t})-f(X^{N, 1 }_{s}) = \int_s^t  b(X^{N, 1 }_{v}, \mu_v^N ) f' (X^{N,1}_{v} ) dv \\
+\int_{]s,t] \times \R_+} \indiq_{\{z\leq f(X^{N,1}_{v-})\}} [f ( X^{N,1}_{v- } + \psi ( X^{N,1}_{v- }, \mu_{v-}^N )  ) - f ( X^{N,1}_{v- } )]   \bar \pi^i (dv, dz) \\
+\sum_{j\neq 1}\int_{]s,t]\times \R_+\times \R} [f(X^{N,1}_{v- } + \frac u{N^{1/\alpha}}) - f(X^{N,1}_{v- } ) ] \indiq_{\{z\leq f(X^{N, j }_{v-})\}}\pi^j(dv,dz,du)  \\
\eqqcolon B^N_{s,t} + \psi^N_{s,t}+ I^N_{s,t}.
\end{multline*}
Using the Lipschitz continuity of $f$ and the boundedness of $f,f', b, \psi,$  we have that 
\[ \E ( | B^N_{s,t} |+ |\psi^N_{s,t}|) \le C |t-s | .\] 
As for the term $I^N_{s,t}$, since $\am<1,$ we use that $f$ is bounded and Lipschitz to obtain  
\[ \left| f \left( v +\frac{ u}{N^{1/\alpha}}\right) - f(v) \right| \le C  \left| f \left( v +\frac{ u}{N^{1/\alpha}}\right) - f(v) \right|^\am \leq  C \frac{ |u|^{\am  }}{ N^{ \am/ \alpha}} ,\]
such that
\[ \E |  I^N_{s,t}  | \le C N^{ 1 - \am / \alpha }  |t-s| .\] 
If $ \alpha > 1, $ we directly study 
$$ X_t^{N, i } - X_s^{N, i } = \int_s^t b(X^{N, 1 }_{v}, \mu_v^N ) dv    +  \frac 1{N^{1/\alpha}}\sum_{j\neq i}\int_{]s,t]\times \R_+\times \R}  u \indiq_{\{z\leq f(X^{N, j }_{v-})\}}\pi^j(dv,dz,du) $$
such that
\begin{multline*}
  |X_t^{N, i } - X_s^{N, i } |^\am \le \\
  C \left( \| b \|_\infty^\am |t-s|^\am + \left| \frac 1{N^{1/\alpha}}\sum_{j\neq i}\int_{]s,t]\times \R_+\times \R}  u \indiq_{\{z\leq f(X^{N, j }_{v-})\}}\pi^j(dv,dz,du)  \right|^\am \right) ,
\end{multline*} 
where we have used that $b$ is bounded. 
  
Since $\nu$ is centered, the second expression on the r.h.s. above (the interaction term) is a martingale. Thus we use the Burkholder-Davis-Gundy inequality in $L^\am$ to obtain
\begin{multline*}
 \E\left[ \left|  \frac 1{N^{1/\alpha}}\sum_{j\neq i}\int_{]s,t]\times \R_+\times \R}  u \indiq_{\{z\leq f(X^{N, j }_{v-})\}}\pi^j(dv,dz,du)\right|^\am  \right]  \leq  \\
  C   
          \E\left[ \left (
         \frac{1}{N^{2/\alpha}}\sum_{j\neq i } \int_{]s,t]\times \R_+ \times \R} u^2  \indiq_{\{z\leq f(X^{N,j}_{v-})\}} \pi^j(dv,dz,du) 
        \right)^{\am/2}\right] .
\end{multline*} 
Recalling Remark \ref{rem:finite_sum} and using the boundedness of $f,$ we obtain 
\begin{multline}\label{eq:similar2}
        \E\left[ \left|  \frac 1{N^{1/\alpha}}\sum_{j\neq i}\int_{]s,t]\times \R_+\times \R}  u \indiq_{\{z\leq f(X^{N, j }_{v-})\}}\pi^j(dv,dz,du)\right|^\am  \right] \\ 
         \leq C \E\left[\frac{1}{N^{\am/\alpha}}\sum_{j\neq i } \int_{]s,t] \times \R_+ \times \R} |u|^\am  \indiq_{\{z\leq f(X^{N,j}_{v-})\} } \pi^j(dv,dz,du) \right] \\
          \leq  C  N^{1- \am/\alpha} \int_{\R} |u|^\am \nu(du) \int_s^t \E\left[f(X^{N,1}_{v})  \right] dv \le 
          C N^{ (1- \am/\alpha)}   |t-s| .  
\end{multline}
Since $ |t- s|^\am \le |t-s| $ (recall that $ |t-s| \le 1 $ and $ \am \geq 1$), this implies item (ii). 
Item (i) follows similarly, since $\E[|  f(X^{N, i }_t)  - f(X^{N, i }_{s }) |] \le \|f\|_{Lip} \E[|  X^{N, i }_t  - X^{N, i }_{s } |] \le \|f\|_{Lip} \left(\E[|  X^{N, i }_t  - X^{N, i }_{s } |^\am] \right)^{1/\am} .$  
\end{proof}

\subsection{Proof of Theorem \ref{prop:representation_finite_syst}}\label{subs:Y}
In this subsection, we give the proof of Theorem \ref{prop:representation_finite_syst}. Let $S^{N,\alpha}$ be the $\alpha$-stable process obtained as in Proposition  \ref{lem:construction_limit_LP} from the sequence of increments $(S^{N,\alpha}_{k\delta,(k+1)\delta})_{k\geq 0}$ given by Proposition \ref {lem:preliminary_results}. Remember that 
\begin{equation*}
A_t^{N}=\frac 1{N^{1/\alpha}}\sum_{j =  1}^N\int_{[0,t]\times\R_+\times\R }u \indiq_{ \{ z \le  f ( X^{N, j}_{{s-}}) \}} \pi^j (ds,dz,du) 
\end{equation*} 
and 
\begin{equation*}
A_t^{N,\delta}=\frac 1{N^{1/\alpha}}\sum_{j =  1}^N\int_{[0,t]\times\R_+\times\R }u \indiq_{ \{ z \le  f ( X^{N, j}_{\tau_{s}}) \}} \pi^j (ds,dz,du) .
\end{equation*} 
Putting  
\begin{equation*}\label{eq:def:R^N1_t}
\begin{split}
R^{N,1}_{t} \coloneqq A^N_t - A^{N,\delta}_t ,
\end{split}
\end{equation*}
\begin{equation*}\label{eq:def:R^N2_t}
    R^{N,2}_t \coloneqq A^{N,\delta}_t - \int_{[0,t]} \left(\mu^N_{\tau_s}(f)\right)^{1/\alpha} d S^{N,\alpha}_s  ,
\end{equation*}
and
\begin{equation*}
R_t^N \coloneqq  R^{N,1}_t + R^{N,2}_t, 
\end{equation*}
we obtain \eqref{eq:atnapprox}. 

\subsection*{Study of  $ R_t^{N, 1 } .$}
Notice that, in case $\alpha > 1,$ $R_t^{N, 1 }$ is a martingale, since $\int_{\R^*}u\nu(du)=0.$ Then, in this case, we can employ the Burkholder-Davis-Gundy inequality in $L^\am$ and similar arguments as those leading to \eqref{eq:similar2}. In case $ \alpha < 1, $ we use instead the sub-additivity of $x\mapsto x^\am$ on $\R_+$ for $\am<\alpha$.  In both cases we obtain
\begin{eqnarray*}
        \E[|R^{N,1}_t |^\am  ]    
         \leq C  N^{1- \am/\alpha} \int_{\R^\ast} |u|^\am \nu(du) \int_0^t \E\left[\left|f(X^{N,j}_{s}) - f(X^{N,j}_{\tau_s}) \right|\right] ds.
 \end{eqnarray*}
Using Proposition \ref{prop:discretis0}, we obtain
\begin{equation}\label{eq:R^N1_t}
    \begin{split}
    \E[|R^{N,1}_t |^\am  ] & \leq C t  \left\{ \begin{array}{ll}
        N^{ ( 1 + \frac{1}{\am}  ) (1  - \am/\alpha)} \delta^{1/\am}, & \mbox{ if } \alpha > 1 \\
        N^{ 2(1 - \am/\alpha )}  \delta & \mbox{ if } \alpha < 1
        \end{array} \right\} .
    \end{split}
\end{equation}
In particular, for $\alpha >1,$ 
\begin{equation}\label{eq:rtn1lq}
\E[|R^{N,1}_t|]\leq \left (\E[|R^{N,1}_t |^\am  ]\right )^{1/\am} \leq C_{t }  N^{ \frac{1}{\am}( 1 + \frac{1}{\am}  ) (1  - \am/\alpha)} \delta^{1/(\am)^2},
\end{equation}
and, for $\alpha<1$, we use $\E[ \norme{R^{N,1}_t}_{d_{\am}} ] \leq \E[|R^{N,1}_t|^{\am}]$.

\subsection*{Study of  $R_t^{N, 2 }  .$}

 Recall \eqref{eq:def:A^N_s_t}, which we will apply for $s=k\delta$ and $t=(k+1)\delta .$ Then, denoting by $\lceil x \rceil$ the upper integer part of $x$, we have 
\begin{equation*}
    \begin{split}
        A^{N,\delta}_t & = \sum_{k=0}^{\lceil \frac{t}{\delta}\rceil - 1} A^{N}_{k\delta, (k+1)\delta} - \frac{1}{N^{1/\alpha}} \sum_{j=1}^N \int_{]t, \lceil \frac{t}{\delta}\rceil \delta ] \times \R_+ \times \R} u \indiq_{\{z \leq f(X^{N,j}_{\left(\lceil \frac{t}{\delta}\rceil -1\right)\delta})\}}\pi^j(ds,dz,du)\\
        & \eqqcolon \sum_{k=0}^{\lceil \frac{t}{\delta}\rceil - 1} A^{N}_{k\delta, (k+1)\delta} - E^1_t .
    \end{split}
\end{equation*}

Also, define
\begin{equation*}
    \begin{split}
        E^2_t \coloneqq \int_{]t, (\lceil \frac{t}{\delta}\rceil +1)\delta ]} \left(\mu^N_{\left(\lceil\frac{t}{\delta}\rceil \right)\delta}(f)\right)^{1/\alpha} d S^{N,\alpha}_s 
    \end{split}
\end{equation*}
and
\begin{equation*}\label{eq:def:E^N,delta_k}
E^{N}_{k\delta, (k+1)\delta }\coloneqq \left|\left(\frac{P^{N}_{k\delta, (k+1)\delta }}{N \delta}\right)^{1/\alpha} - (\mu^N_{k\delta}(f))^{1/\alpha}  \right|
| S^{N,\alpha}_{(k+1)\delta} - S^{N,\alpha}_{k\delta} |  ,
\end{equation*} 
where $P^{N}_{k\delta, (k+1)\delta}$ was defined in \eqref{eq:def:P^N_s_t}.

Then, for $1<\alpha<2$, we can write
\begin{eqnarray}\label{eq:R^N2_t:bound_1}
& & \E[| R^{N,2}_t| ]  \leq
\E\left[\left| \sum_{k=0}^{\lceil \frac{t}{\delta}\rceil-1} \left\{ A^{N}_{k\delta, (k+1)\delta }  - \int_{]k\delta, (k+1)\delta ]}\left(\mu^N_{k\delta}(f)\right)^{1/\alpha}d S^{N,\alpha}_s  \right\} -E^1_t + E^2_t  \right|\right] \nonumber \\
& & \leq \E\left[ \sum_{k=0}^{\lceil \frac{t}{\delta}\rceil -1} \left|A^N_{k\delta, (k+1)\delta} - \left(\frac{P^{N}_{k\delta, (k+1)\delta }}{N}\right)^{1/\alpha} S^{N,\alpha}_{k\delta, (k+1)\delta} \right.\right. \nonumber \\
&      & \left.\left. + \left(\frac{P^{N}_{k\delta, (k+1)\delta }}{N}\right)^{1/\alpha} S^{N,\alpha}_{k\delta, (k+1)\delta} - \int_{]k\delta, (k+1)\delta ]}\left(\mu^N_{k\delta}(f)\right)^{1/\alpha}d S^{N,\alpha}_s  \right| \right] \nonumber  \\
&      & \quad \quad \quad  \quad \quad \quad \quad \quad \quad \quad \quad \quad  \quad \quad \quad \quad \quad \quad \quad \quad \quad  \quad \quad \quad+ \E[|E^1_t|] + \E[|E^2_t|] \nonumber \\
&      & \leq \E\left[ \sum_{k=0}^{\lceil \frac{t}{\delta}\rceil -1} \left|A^N_{k\delta, (k+1)\delta} - \left(\frac{P^{N}_{k\delta, (k+1)\delta }}{N}\right)^{1/\alpha} S^{N,\alpha}_{k\delta, (k+1)\delta} \right|\right] \nonumber \\
& & + \E\left[ \sum_{k=0}^{\lceil \frac{t}{\delta}\rceil - 1} \left|\left(\frac{P^{N}_{k\delta, (k+1)\delta }}{N \delta}\right)^{1/\alpha} - (\mu^N_{k\delta}(f))^{1/\alpha}  \right|
| S^{N,\alpha}_{(k+1)\delta} - S^{N,\alpha}_{k\delta} | \right] \nonumber \\
 && \quad \quad \quad  \quad \quad \quad \quad \quad \quad \quad \quad \quad  \quad \quad \quad \quad \quad \quad \quad \quad \quad  \quad \quad \quad+ \E[|E^1_t|] + \E[|E^2_t|] \nonumber \\
& & \leq \left(\lceil \frac{t}{\delta}\rceil - 1\right) C \delta^{1/\alpha} g(N \delta) + \sum_{k=0}^{\lceil \frac{t}{\delta}\rceil - 1} \E[E^{N}_{k\delta, (k+1)\delta }] +\E[|E^1_t|] + \E[|E^2_t|] ,
\end{eqnarray}
where we have used that, thanks to \eqref{eq:as}, almost surely, $ \delta^{1/\alpha} S^{N,\alpha}_{k\delta, (k+1)\delta}  =  S^{N,\alpha}_{(k+1)\delta} - S^{N,\alpha}_{k\delta}$, and the last inequality follows straight from Proposition \ref {lem:preliminary_results} $i)$.

For $\alpha<1$, using the sub-additivity of the function $ x \mapsto x^\am $ on $\R_+ $ for $ \am < \alpha < 1$,  
we can write 
\begin{eqnarray}\label{eq:R^N2_t:bound_1:alpha<1}
& & \E[ \norme{ R^{N,2}_t}_{d_{\am}} ]  \leq 
\E\left[ \sum_{k=0}^{\lceil \frac{t}{\delta}\rceil -1} \left|A^N_{k\delta, (k+1)\delta} - \left(\frac{P^{N}_{k\delta, (k+1)\delta }}{N}\right)^{1/\alpha} S^{N,\alpha}_{k\delta, (k+1)\delta} \right| \right.  \\
& & \left. \quad \quad \quad \quad \quad \quad \quad \quad \quad \quad 
\wedge \sum_{k=0}^{\lceil \frac{t}{\delta}\rceil -1} \left|A^N_{k\delta, (k+1)\delta} - \left(\frac{P^{N}_{k\delta, (k+1)\delta }}{N}\right)^{1/\alpha} S^{N,\alpha}_{k\delta, (k+1)\delta} \right|^{\am} \right] \nonumber \\
& & + \E\left[ \sum_{k=0}^{\lceil \frac{t}{\delta}\rceil - 1} \left|\left(\frac{P^{N}_{k\delta, (k+1)\delta }}{N \delta}\right)^{1/\alpha} - (\mu^N_{k\delta}(f))^{1/\alpha}  \right|^{\am}
| S^{N,\alpha}_{(k+1)\delta} - S^{N,\alpha}_{k\delta} |^{\am} \right] \nonumber \\
 && \quad \quad \quad  \quad \quad \quad \quad \quad \quad \quad \quad \quad  \quad \quad \quad \quad \quad \quad \quad \quad \quad  \quad \quad \quad+ \E[|E^1_t|^{\am}] + \E[|E^2_t|^{\am}] \nonumber\\
& & \leq \left( \lceil \frac{t}{\delta} \rceil -1 \right) C \delta^{\am/\alpha} g(N\delta) + \sum_{k=0}^{\lceil \frac{t}{\delta}\rceil - 1} \E[(E^{N}_{k\delta, (k+1)\delta })^\am] +\E[|E^1_t|^\am] + \E[|E^2_t|^\am] .  \nonumber
\end{eqnarray}
In the last inequality we employed Proposition \ref{lem:preliminary_results} $ii)$, and we used that the exponential factor is negligible compared to the other terms. 

Now, using the same arguments as those used in the proof of Proposition \ref{prop:discretis0}, we have in both cases $\alpha<1$ and $1<\alpha<2,$
\begin{equation}\label{eq:E^1_t}
\E[ | E^1_t |^{\am \wedge 1}] \leq C  \left\{ 
\begin{array}{ll}
N^{1/\am - 1/\alpha } \delta^{1/\am} & \mbox{ if } \alpha > 1 \\
N^{ 1 - \am/ \alpha} \delta & \mbox{ if } \alpha < 1  \\
\end{array} \right\} .         
\end{equation} 
Moreover, since $f$ is bounded and $S^{N,\alpha}_{s}-S^{N,\alpha}_r \overset{d}{=} |s-r|^{1/\alpha} S^\alpha$, 
\begin{equation}\label{eq:E2t}
    \begin{split}
        \E[| E^2_t|^{\am \wedge 1} ] & = \E\left[\left|\left(\mu^N_{\left(\lceil\frac{t}{\delta}\rceil -1\right)\delta}(f) \right)^{1/\alpha} \right|^{\am \wedge 1} \left| \int_{]t, \lceil \frac{t}{\delta}\rceil\delta]} d S^{N,\alpha}_s \right|^{\am \wedge 1} \right]
       \\
       & \leq C \E[|S^{N,\alpha}_{\lceil \frac{t}{\delta}\rceil\delta} - S^{N,\alpha}_t|^{\am \wedge 1}] \leq C \delta^{^{(\am \wedge 1)}/\alpha}.
    \end{split}
\end{equation}
Last, we use once more deviation inequalities to deal with the (conditional) Poisson random variables: as before, since $f$ is bounded and lowerbounded, for the event  
\begin{equation*}
        G \coloneqq \left\{\frac{1}{2} \delta N \underline{f}  \le P^{N}_{k\delta, (k+1)\delta} \leq 2  \delta N \|f\|_\infty  \right\} , \mbox{ we have }  \mathbf{P}(G^c ) \leq c e^{-C \delta N}. 
\end{equation*}
Notice that  $0 < \underline f  < \mu^N_{k\delta}(f) \le \|f\|_\infty . $ So we can use the Lipschitz property of $z \mapsto z^{1/\alpha}$ on $ [ \frac12 \underline f , \infty [ $ in case $ \alpha > 1 $ and on $ [ 0, 2 \|f\|_\infty ] $ in case $ \alpha < 1 ,$ Jensen's inequality  and the fact that $P^{N}_{k\delta, (k+1)\delta}$ is conditionally  Poisson distributed together with the boundedness of $f$  to deduce that 
\begin{eqnarray*}
&&\E\left[\left|\left(\frac{P^{N}_{k\delta, (k+1)\delta}}{N \delta }\right)^{1/\alpha} - \left(\mu^N_{k\delta}(f)\right)^{1/\alpha}\right|^{ \am \wedge 1 }  \indiq_G  \right] \le C   \E\left[\left|\frac{P^{N}_{k\delta, (k+1)\delta}}{N \delta } - \mu^N_{k\delta}(f)\right|^{ \am \wedge 1 } \indiq_G  \right]\\
&&\le  C   \E\left[\left|\frac{P^{N}_{k\delta, (k+1)\delta}}{N \delta } - \mu^N_{k\delta}(f)\right| \indiq_G  \right]^{ \am \wedge 1 } \le C {\E\left[\left|\frac{P^{N}_{k\delta, (k+1)\delta}}{N \delta } - \mu^N_{k\delta}(f)\right|^2 \right]}^{ \frac12 ( \am \wedge 1 ) } \\
&&\le  C \left( {\frac{1}{N^2 \delta^2}Var\left[Pois\left( N \delta \mu^N_{k\delta}(f)\right)\right]}\right)^{ \frac12 ( \am \wedge 1 ) } \le  C (N \delta)^{-( \am \wedge 1 )/2}. 
\end{eqnarray*}
Moreover, using H\"older's inequality, bounding  $[ P^{N}_{k\delta, (k+1)\delta} / (N \delta) ]^{ (\am \wedge 1)/ \alpha } $ in terms of its first moment, and using that $\mu^N_{k\delta}(f) \le \|f\|_\infty, $ 
\[  \E\left[\left|\left(\frac{P^{N}_{k\delta, (k+1)\delta}}{N \delta }\right)^{1/\alpha} - \left(\mu^N_{k\delta}(f)\right)^{1/\alpha}\right|^{ \am \wedge 1 }  \indiq_{ G^c}  \right]  \le  C e^{-c\delta N}.\]
Therefore, using the independence of $S^{N,\alpha}_{(k+1)\delta}- S^{N,\alpha}_{k\delta}$ of $\F_{k\delta}$ and of $P^{N}_{k\delta, (k+1)\delta},$ we conclude that 
\begin{equation}\label{eq:bound:ENk2delta}
    \begin{split}
        \E[(E^{N}_{k\delta, (k+1)\delta})^{\am \wedge 1}] & =  \E\left[\left| \left(\frac{P^{N}_{k\delta, (k+1)\delta}}{N\delta}\right)^{1/\alpha} - (\mu^N_{k\delta}(f))^{1/\alpha} \right|^{ \am \wedge 1 } \right] \E \left[ |S^{N,\alpha}_{(k+1)\delta}- S^{N,\alpha}_{k\delta}|^{ \am \wedge 1 } \right] \\
        & \leq  C (N \delta)^{-( \am \wedge 1 )/2}  \delta^{(\am \wedge 1)/\alpha}.
    \end{split}
\end{equation}
Overall, putting together \eqref{eq:R^N2_t:bound_1} (\eqref{eq:R^N2_t:bound_1:alpha<1} if $\alpha<1$), \eqref{eq:E^1_t}, \eqref{eq:E2t} and \eqref{eq:bound:ENk2delta}, 
we have in the case $\alpha>1$,
\begin{equation*}
\E[ | R^{N,2}_t |  ] \leq C_{t} \left[ \lceil\frac{t}{\delta}\rceil \delta^{1/\alpha} \left(    g(N \delta)+  (N \delta)^{- 1/2}   \right) + N^{1/\am - 1/\alpha } \delta^{1/\am } + \delta^{1  /\alpha}\right] 
\end{equation*}
and 
\begin{equation*}
  \E[  \norme{R^{N,2}_t}_{d_{\am}}] \leq  C_{t} \left[
\lceil\frac{t}{\delta}\rceil \left(  \delta^{\am /\alpha}  g(N \delta)+  (N \delta)^{-\am /2}  \delta^{\am  /\alpha} \right) + N^{ 1 - \am / \alpha} \delta  +\delta^{\am  /\alpha} \right]  
\end{equation*}
for $\alpha<1$. 

Recall the control on $ R^{N, 1 }_t$ obtained in \eqref{eq:rtn1lq} above in case $ \alpha > \am > 1.$ We clearly have that $N^{1/\am - 1/\alpha } \delta^{1/\am} \le  N^{\left( 1 + \frac{1}{\am }\right) \left( 1 - \frac{\am }{\alpha}\right) } \delta^{\frac{1}{\am }} .$ Suppose now that $ \delta = \delta(N) $ is small enough, such that the latter expression is smaller than $1.$ Since $\am > 1,$ this implies that 
\[ N^{1/\am - 1/\alpha } \delta^{1/\am}  \le   N^{\left( 1 + \frac{1}{\am }\right) \left( 1 - \frac{\am }{\alpha}\right) } \delta^{\frac{1}{\am }} \le \left(  N^{\left( 1 + \frac{1}{\am }\right) \left( 1 - \frac{\am }{\alpha}\right) } \delta^{\frac{1}{\am }}\right)^{1/\am } \le r_t(N, \delta) . \]
Moreover, for 
$ \alpha < 1, $ recalling \eqref{eq:R^N1_t}, clearly  $  N^{ 1 - \am / \alpha} \delta   \le  N^{ 2 ( 1 - \frac{\am }{\alpha}) } \delta \le r_t(N,\delta)$, which allows to conclude. 
$\qed$
\section{Convergence to the limit system}\label{sec:convergencelim}
\subsection{Some technical results}\label{eq:sectechique}
We provide some technical lemmas that we 
state in this subsection for generic processes $X $ and $ \tilde X $ with associated measures $ \mu $ and $ \tilde \mu $ which can be either the associated empirical distributions or the respective conditional laws given $ S^\alpha .$ 
In the sequel, $ M $  denotes the  jump measure of $S^\alpha $ and, as before, $ T_K = \inf \{ t \geq 0 : | \Delta S_t^\alpha | > K \}.$ 
Moreover, we will always assume that Assumptions \ref{ass:b}--\ref{ass:nu0+} are satisfied. 

 \begin{lemma}\label{lem:controlM}
Let 
\[
M_t(\mu,\tilde \mu) \coloneqq \int_{[0, t ] \times \R^* }   \left(\mu^{1/ \alpha}_{s-}(f)  - \tilde \mu^{1/ \alpha}_{s-}(f)\right)z    \tilde M ( ds, dz)\quad  \mbox{ if } \quad\alpha > 1  ,
\] 
and
\[
M_t(\mu,\tilde \mu) \coloneqq \int_{[0, t ] \times \R^* }   \left( \mu^{1/ \alpha}_{s-}(f)  - \tilde \mu^{1/ \alpha}_{s-}(f)\right)z    M ( ds, dz)\quad  \mbox{ if } \quad\alpha < 1  .
\] 
\begin{description}
\item{i)} For all $1<\alpha<\ap\leq 2, $ we have for any $ K >  0, $  
\begin{multline}\label{eq:controlI1}
 \E\left [  |M_t(\mu,\tilde \mu)|^{\ap}\indiq_{\{t < T_K\}}\right ]  \le C \frac{K^{\ap-\alpha}}{\ap - \alpha}  \int_0^t \E \left [   |  \mu_{s}(f)- \tilde  \mu_{s}(f) |^{\ap}\indiq_{\{s < T_K\}}\right] ds \\
 \leq  C \frac{K^{\ap-\alpha}}{ \ap - \alpha} \int_0^t \E \left [   W_1(  \mu_{s}, \tilde\mu_{s} )\indiq_{\{s < T_K\}}\right] ds .
\end{multline} 
 \item{ii)} For all $0< \am < \alpha<1,$ we have for any $ K > 0,$ 
 \begin{multline}\label{eq:controlI2}
 \E\left [  |M_t(\mu,\tilde \mu)|\indiq_{\{t < T_K\}}\right ]  \le C \frac{K^{1-\alpha}}{ 1- \alpha} \int_0^t\E\left[|\mu_s(f) - \tilde \mu_s(f)|\indiq_{\{s < T_K\}}\right]ds \\
\le C \frac{K^{1-\alpha}}{ 1- \alpha}  \int_0^t\E\left[W_{d_\am}(\mu_s,\tilde \mu_s)\indiq_{\{s < T_K\}}\right]ds .
 \end{multline}
 \end{description}
\end{lemma}

\begin{proof} 
We write 
\[ M_t(\mu,\tilde \mu) \eqqcolon M_t^{1}(\mu,\tilde \mu)+M_t^{2}(\mu,\tilde \mu), \]
where $M_t^1(\mu,\tilde \mu)$ corresponds to the integral over $[0,t]\times  B^\ast_K$ and $M_t^2(\mu,\tilde \mu)$ to that over $[0,t]\times  B^c_K.$ 
Let  us first consider the case $1<\alpha<\ap \leq  2.$ 
Using first the Burkholder-Davis-Gundy inequality and then the sub-additivity of $x\to x^{\ap/2},$ we obtain
\begin{multline*}
\E[|M_t^{1}(\mu,\tilde \mu)|^{\ap}\indiq_{\{t < T_K\}}] \le  \E\left[ \left |\int_{[0,t\wedge T_K]\times B^\ast_K}\left(\mu_s^{1/\alpha}(f)-\tilde\mu^{1/\alpha}_{s}(f)\right)z\tilde M(ds,dz)\right |^{\ap} \right]\\
\leq C \E\left[ \left (\int_{[0,t\wedge T_K]\times B^\ast_K}\left(\mu_s^{1/\alpha}(f)-\tilde\mu_{s}^{1/\alpha}(f)\right)^2 z^2M(ds,dz)\right )^{\ap/2} \right]\\
\leq  C \int_0^t\E\left[|\mu_s^{1/\alpha}(f)-\tilde\mu_{s}^{1/\alpha}(f)|^{\ap}\indiq_{\{s < T_K\}}\right]ds\int_{B^\ast_K} \frac {|z|^{\ap}}{|z|^{1+\alpha}}dz \\
=  C\frac{K^{\ap-\alpha}}{ \ap - \alpha} \int_0^t\E[|\mu_s^{1/\alpha}(f)-\tilde\mu_{s}^{ 1/\alpha}(f)|^{\ap}\indiq_{\{s < T_K\}}]ds. 
\end{multline*}
As for $M_t^{2}(\mu,\tilde \mu)$, using Jensen's inequality,  
\begin{multline*}
\E[|M_t^{2}(\mu,\tilde \mu)|^{\ap}\indiq_{\{t < T_K\}}] \le  \E\left[\left|\int_0^t(\mu_s^{1/\alpha}(f)-\tilde\mu_s^{1/\alpha}(f))\indiq_{\{s < T_K\}}ds\int_{B^c_K}\frac{z}{|z|^{1+\alpha}} dz\right|^{\ap}\right]\\
\leq C_t\int_0^t\E\left[ |\mu_s^{ 1/\alpha}(f)-\tilde\mu_s^{ 1/\alpha}(f)|^{\ap}\indiq_{\{s < T_K\}}\right]ds .
\end{multline*}

 Now, in both previous bounds we use the Lipschitz property of $z\mapsto z^{1/\alpha}$ on $[\underline{f}, \infty)$ to obtain the first inequality in \eqref{eq:controlI1} and we use \eqref{eq:controlmuf} to obtain the second bound in \eqref{eq:controlI1}: 
\begin{multline*}
\E\left[|M_t(\mu,\tilde \mu)|^{\ap}\indiq_{\{t < T_K\}}\right]\leq C \E\left[|M_t^1(\mu,\tilde \mu)|^{\ap}\indiq_{\{t < T_K\}}\right]+C \E\left[|M_t^2(\mu,\tilde \mu)|^{\ap}\indiq_{\{t < T_K\}}\right]\\
\leq  C_t \frac{K^{\ap-\alpha}}{ \ap - \alpha} \int_0^t\E[|\mu_s(f)-\tilde\mu_{s}(f)|^{\ap}\indiq_{\{s < T_K\}}]ds \leq C_t \frac{K^{\ap-\alpha}}{ \ap - \alpha}  \int_0^t \E \left [   W_1(  \mu_{s}, \tilde \mu_{s} )\indiq_{\{s < T_K\}}\right] ds .
\end{multline*}

Let now $0<\alpha<1$ so that $M_t^{2}(\mu,\tilde \mu),$ being a non-compensated integral, equals zero on the event $ \{ t < T_K \}.$  Using that $z\to z^{1/\alpha}$ is Lipschitz on $[0,\norme{f}_\infty],$ we obtain  
\begin{eqnarray*}
\E[|M_t(\mu,\tilde \mu)|\indiq_{\{t < T_K\}}]&=&\E[|M_t^1(\mu,\tilde \mu)|\indiq_{\{t < T_K\}}]\\
&\leq& \E\left[ \int_{[0,t\wedge T_K]\times B^\ast_K}\left|\mu_s^{1/\alpha}(f)-\tilde\mu^{1/\alpha}_{s}(f)\right | |z| M(ds,dz) \right]\\
&\leq& C\E\left[ \int_{[0,t]}\left|\mu_s^{1/\alpha}(f)-\tilde\mu_{s}^{1/\alpha}(f)\right |\indiq_{\{s<T_K\}} ds \int_{B^\ast_K}\frac{|z|}{|z|^{\alpha+1}} dz \right]\\
&\leq &C \frac{K^{1-\alpha}}{1 - \alpha} \E\left[ \int_{[0,t]}\left|\mu_s(f)-\tilde\mu_{s}(f)\right | \indiq_{\{s < T_K\}} ds    \right] .
\end{eqnarray*}
Finally, using \eqref{eq:controlmuf2}, this last expression is in turn upper bounded by 
\[ C\frac{K^{1-\alpha}}{1 - \alpha}  \int_0^t\E\left[W_{d_\am}(\mu_s,\tilde \mu_s)\indiq_{\{s < T_K\}}\right]ds.\] 
\end{proof}

\begin{lemma}\label{lem:controlb}
Let 
\[
B_t (X, \tilde X, \mu , \tilde \mu) \coloneqq \int_0^t[   b ( X_{s} ,  \mu_{s}) -  b ( \tilde X_{s}, \tilde \mu_s  )] ds .
\] 
\begin{description}
\item{i)} If $1<\alpha<2,$ then for any $1<\alpha<\ap \le 2$ and for any $ K > 0, $ 
\begin{multline*} 
\E[|B_t (X, \tilde X, \mu , \tilde \mu)|^{\ap}\indiq_{\{t < T_K\}}]\leq   \\
  C_t   \left\{ \begin{array}{ll}
\E \int_0^t  \indiq_{\{s < T_K\}} \left(  |  X_s - \tilde X_s |^{\ap} + W_{1}^{\ap} ( \mu_s, \tilde \mu_s ) \right) ds .\\
\E \int_0^t  \indiq_{\{s < T_K\}} \left( | X_s - \tilde X_s | + W_{1} ( \mu_s, \tilde \mu_s )\right) ds.
\end{array}
\right.
\end{multline*}
\item{ii)} If $0<\alpha<1,$ then for any $0<\am<\alpha$ and for any $ K > 0, $
\begin{equation*}
\E[|B_t (X, \tilde X, \mu , \tilde \mu)|\indiq_{\{t < T_K \}}]\leq C 
\E\int_0^t \indiq_{\{s<T_K\}} (d_\am(X_s,\tilde X_s)+W_{d_\am}(\mu_s,\tilde\mu_s))ds.
\end{equation*} 

\end{description}
In particular, if $X=X^{N,i},$ $\tilde X=\tilde X^{N,i},$ where $(X^{N,i})_{i=1}^N$, $(\tilde X^{N,i})_{i=1}^N$ are two systems  defined on the same probability space and $(X^{N,i},\tilde X^{N,i})_{i=1}^N$ is exchangeable, and if moreover $\mu=\mu^{N,X}$, $\tilde\mu=\tilde\mu^{N,X}$ are the empirical measures of $X$ and $\tilde X,$ then, for any $1<\alpha<\ap \le 2$ and for any $ K > 0, $
 \begin{equation}\label{eq:b2}
\E[|B_t (X, \tilde X, \mu , \tilde \mu)|^{\ap}\indiq_{\{t < T_K\}}]\leq C  \left\{
\begin{array}{ll}
\int_0^t\E[|X_s-\tilde X_s|^{\ap}\indiq_{\{s < T_K\}}]ds,\\
\int_0^t \E[|X_s - \tilde X_s|\indiq_{\{s<T_K\}}] ds ,
\end{array}
\right.
\end{equation}
and, for any $0<\am<\alpha < 1 $ and for any $ K > 0, $
\begin{equation}\label{eq:b3}
\E[|B_t (X, \tilde X, \mu , \tilde \mu)|\indiq_{\{t < T_K\}}]\leq C
\int_0^t\E[d_\am(X_s, \tilde X_s)\indiq_{\{s < T_K\}}]ds .
\end{equation}
\end{lemma}

\begin{proof}
The inequalities of item i) and item  ii) follow  immediately from the Lipschitz property of $b,$ using moreover the fact that $b$ is also bounded  in case $ \alpha > 1 .$ 
 Equations \eqref{eq:b2} and \eqref{eq:b3}  follow from the  fact that
$\frac 1N\sum_{i=1}^N\delta_{(X^{N,i}_s,\tilde X^{N,i}_s)}$ is a coupling of   $\frac 1N\sum_{i=1}^N\delta_{X^{N,i}_s}$ and $\frac 1N\sum_{i=1}^N\delta_{\tilde X_s^{N,i} }$ and the exchangeability of $(X^{N,i},\tilde X^{N,i})_{i=1}^N.$
Of course both inequalities hold also without indicator, but in the sequel we need them in this form.
\end{proof}

\begin{lemma}\label{lem:controlpsi}
For $\alpha<1$, let 
\[
\Psi_t (X, \tilde X, \mu , \tilde \mu) \coloneqq \int_{[0,t]\times \R_+} \left[   \psi(X_{s-},\mu_{s-}) \indiq_{\{z\leq f(X_{s-})\}} - \psi(\tilde X_{s-},\tilde\mu_{v-}) \indiq_{\{z\leq f(\tilde X_{s-})\}} \right] \bar\pi(ds,dz) .
\] 
Then, for any $0<\am<\alpha$ and any $ K > 0, $ 
\begin{equation}\label{eq:controlpsi1}
\E[|\Psi_t(X,\tilde X, \mu, \tilde \mu)|\indiq_{\{t < T_K\}}] \leq C \int_0^t \E\left[ ( d_\am(X_s,\tilde X_s) + W_{d_\am}(\mu_s, \tilde \mu_s)) \indiq_{\{t<T_K\}}\right] ds.
\end{equation}
In particular, if $X=X^{N,i},$ $\tilde X=\tilde X^{N,i},$ where $(X^{N,i})_{i=1}^N$, $(\tilde X^{N,i})_{i=1}^N$ are two systems  defined on the same probability space and $(X^{N,i},\tilde X^{N,i})_{i=1}^N$ is exchangeable, and if moreover $\mu=\mu^{N,X}$, $\tilde\mu=\tilde\mu^{N,X}$ are the empirical measures of $X$ and $\tilde X,$
\begin{equation}\label{eq:controlpsi2}
\E[|\Psi_t(X,\tilde X, \mu, \tilde \mu)|\indiq_{\{t < T_K\}}] \leq C \int_0^t \E\left[  d_\am(X_s,\tilde X_s)  \indiq_{\{t<T_K\}}\right] ds.
\end{equation}
\end{lemma}

\begin{proof}
The proof is analogous to the one of Lemma \ref{lem:controlb}, using Assumption \ref{ass:bqLipp+psiqLip} and using the boundedness and Lipschitz continuity of $f$ and the Kantorovitch-Rubinstein duality. 

\end{proof}

\subsection{Introducing an auxiliary process}
In what follows, to clearly distinguish the empirical measures of the respective auxiliary processes that we shall consider, we write $\mu^{N,X}$  
 for the empirical measure of the finite system $(X^{N,i})_{1\leq i \leq N}$ (see \eqref{eq:micro:dyn}). %, and we write $\mu^{N,Y}$ for the empirical measure of the auxiliary system $(Y^{N,i})_{1\leq i \leq N},$ which,} 
 
Based on Theorem \ref{prop:representation_finite_syst}, we
introduce for all $N\in \N^\ast$, $i=1,\dots,N,$ the auxiliary process
\begin{multline}\label{eq:def:Y^{N,i}}
Y^{N,i}_t = X^{i}_0 + \int_0^t b(X^{N,i}_s, \mu^{N,X}_s) ds +\int_{[0,t]\times\R_+} \psi (X^{N, i}_{s-}, \mu^{N,X}_{{ s-}} )  \indiq_{ \{ z \le  f ( X^{N, i}_{s-}) \}} \bar\pi^i (ds,dz) \\
+  \int_{[0,t]} \left(\mu^{N,X}_{s-}(f)\right)^{1/\alpha} d S^{N,\alpha}_s ,
\end{multline}
where we recall that $ \psi ( \cdot ) \equiv 0 $ in case $ \alpha > 1.$ Let $ T_K^N = \inf \{ t \geq 0 : | \Delta S_t^{N, \alpha} | > K \} .$ In what follows, $C_{ t } $ denotes a constant (that may change from line to line) which is non-decreasing as a function of $t$,  and $r_t(N,\delta)$ is given in Theorem \ref{thm:prop:chaos}.

\begin{proposition}\label{prop:secondrepr}
Grant Assumptions \ref{ass:b}--\ref{ass:nu0+}. Then, on an extension of $(\Omega, \mathcal{A},\mathbf{P})$ depending on $N$ and $ \delta, $ 
\begin{description}
\item [i)]
if $\alpha > 1,$ then we have for any $\am ,\ap$ such that  $1<\am <\alpha<\ap<2,$
\[
  \E[|X_t^{N,i}-Y^{N,i}_t|\indiq_{\{t < T^N_K\}}]\leq C_t r_t^{X,Y},
\]
where
\[
  r_t^{X,Y}\coloneqq \left( \frac{K^{\ap-\alpha}}{ \ap - \alpha }\right)^{1/ \ap }  \left( N^{1 - \am/\alpha} \delta \right)^{1/\ap}  + r_t(N, \delta);
\]
\item [ii)] if $0<\alpha<1,$ then we have for any $\am $ such that $0<\am <\alpha<1,$ for any $ K > 0, $ 
\begin{equation*}
  \E[  d_\am (X_t^{N,i}, Y^{N,i}_t) \indiq_{\{t < T^N_K\}}]\leq  C_t  \frac{K^{1-\alpha}}{1-\alpha} r_t(N, \delta) .
 \end{equation*}
\end{description}
\end{proposition}

\begin{proof}[Proof of Proposition \ref{prop:secondrepr}] 
Clearly, \[ X_t^{N,i }=Y_t^{N, i } + R_t^N - E_t^{N, i }-M_t^{N},\] 
where $R_t^N$ is defined in Theorem  \ref{prop:representation_finite_syst}, $M_t^N\coloneqq M_t( \mu^{N,X},\mu_{\tau}^{N,X})$ as in Lemma \ref{lem:controlM}  and 
\begin{equation}\label{eq:def:E^Ni_t}
E^{N,i}_t \coloneqq  \int_{[0,t]\times \R_+ \times\R} \frac{u}{N^{1/\alpha}} \indiq_{\{z\leq f(X^{N,i}_{s-})\}} \pi^i(ds,dz,du).
\end{equation}
Since $ f$ is bounded, in case $ \alpha > 1,$ we upperbound 
\begin{equation}\label{eq:E^Ni_t11}
 \E[| E^{N,i}_t |] \leq \frac{C}{N^{1/\alpha}}\E\left[\int_{[0,t]\times \R_+ \times \R} |u| \indiq_{ \{z \leq \| f\|_\infty \}  } \pi^i(ds,dz,du)  \right] \leq  C t N^{-1/\alpha},
\end{equation}
whereas if $0<\am <\alpha<1,$ using sub-additivity,
\begin{equation}\label{eq:E^Ni_t12}
 \E[| E^{N,i}_t |^\am] \leq \frac{C}{N^{\am/\alpha}}\E\left[\int_{[0,t]\times \R_+ \times \R} |u|^\am \indiq_{ \{z \leq \| f\|_\infty \}  } \pi^i(ds,dz,du)  \right] \leq  C t N^{-\am /\alpha},
\end{equation}
Finally, since $N\delta\to\infty,$ it is easy to see that both upper bounds \eqref{eq:E^Ni_t11},\eqref{eq:E^Ni_t12} are bounded by $C r_t(N,\delta).$

Let $1<\am<\alpha<\ap<2.$ Then using Lemma \ref{lem:controlM},  Jensen's inequality, exchangeability, the boundedness and Lipschitz continuity of $f$ and item ii) of Proposition \ref{prop:discretis0},  
\begin{eqnarray}\label{eq:MM}
\E\left [|M^N_t|^{\ap}\indiq_{\{t < T^N_K\}}\right] &\leq & C \frac{K^{\ap-\alpha}}{\ap-\alpha}\int_0^t\E\left[\left|\mu_s^{N,X}(f)-\mu_{\tau_s}^{N,X}(f)\right|^{\ap}\right]ds
\\ \nonumber
&\leq&  C \frac{K^{\ap-\alpha}}{\ap-\alpha} \int_0^t\E[|f(X^{N,1}_s)-f(X^{N,1}_{\tau_s})|^{\ap}]ds \\ \nonumber 
&\leq &  C \frac{K^{\ap-\alpha}}{\ap-\alpha} \int_0^t \E[|X^{N,1}_s- X^{N,1}_{\tau_s} |^\am ] ds \\ \nonumber
&\leq &  C_t \frac{K^{\ap-\alpha}}{\ap-\alpha}   N^{ 1 - \am/ \alpha} \delta 
.\nonumber
\end{eqnarray}
To go from the second line above to the third line, we have used that, since $f$ is bounded, $ |f(X^{N,1}_s)-f(X^{N,1}_{\tau_s})|^{\ap} \le C |f(X^{N,1}_s)-f(X^{N,1}_{\tau_s})|^{\am} ,$ which is in turn upper bounded by 
$ C \| f\|_{Lip} |X^{N,1}_s-X^{N,1}_{\tau_s}|^{\am}.$

Remember that Theorem \ref {prop:representation_finite_syst} gives
$\E[ | R_t^{N } | ] 
\leq  C_{t} r_t(N, \delta)$ in case $\alpha > 1.$ 
Using this bound and collecting \eqref{eq:MM}, \eqref{eq:E^Ni_t11} we obtain the claim $i).$
Whereas if $0<\alpha<1,$ the claim $ii)$ follows from upper bounding $ d_\am (X_t^{N,i}, Y^{N,i}_t ) \le \norme{R_t^N}_{d_\am} + |E_t^{N, i } |^\am + | M_t^N | , $ using then Theorem \ref {prop:representation_finite_syst}, \eqref{eq:E^Ni_t12}, Lemma \ref{lem:controlM}, exchangeability and Proposition \ref{prop:discretis0}: 
\begin{multline*}\label{eq:M}
\E\left [|M_t^N| \indiq_{\{t < T^N_K\}}\right]\leq  C \frac{K^{1-\alpha}}{1-\alpha} \int_0^t\E\left[\left|\mu_s^{N,X}(f)-\mu_{\tau_s}^{N,X}(f)\right|\right]ds  \\
\leq  C \frac{K^{1-\alpha}}{1-\alpha} \int_0^t\E[|f(X^{N,1}_s)-f(X^{N,1}_{\tau_s})|]ds\leq  C_t \frac{K^{1-\alpha}}{1-\alpha} N^{ 1 - \am/ \alpha } \delta ,
\end{multline*}
 noticing that this last term is upper bounded by $ \frac{K^{1-\alpha}}{1-\alpha}r_t(N,\delta).$
\end{proof}

\begin{proposition}\label{prop:YtoY}
Grant Assumptions \ref{ass:b}--\ref{ass:nu0+}.  Let $Y^N$ given by \eqref{eq:def:Y^{N,i}} and  
 $r_t^{X,Y}$ given by Proposition \ref{prop:secondrepr}. Write $\mu^{N,Y}$ for the empirical measure of $ Y^N.$ Then $Y^N$ can be represented as 
\begin{multline*}
Y^{N,i}_t = X^{i}_0 + \int_0^t b(Y^{N,i}_s, \mu^{N,Y}_s) ds + \int_{[0, t ] \times \R_+} \psi ( Y^{N,i}_{s-}, \mu^{N,Y}_{s-}) \indiq_{\{z\leq f(Y^{N,i}_{s-})\}} \bar \pi^i (ds, dz) \\
+  \int_{[0,t]} \left(\mu^{N,Y}_{s-}(f)\right)^{1/\alpha} d S^{N,\alpha}_s +R_t^{N,Y},\quad\mbox{and}\quad
\end{multline*}
\begin{description}
\item{i)}
if $1<\alpha<2,$ then for any $1<\alpha<\ap\leq2,$
\begin{equation*}
\E[|R_t^{N,Y}|^{\ap}\indiq_{\{t < T^N_K\}}]\leq    C_t \frac{K^{\ap-\alpha}}{\ap-\alpha} r_t^{X,Y}   ;    
\end{equation*}
\item {ii)} if $0<\alpha<1,$ then for any $ K > 0,$ 
\begin{equation*}
\E[|R_t^{N,Y}|\indiq_{\{t < T^N_K\}}]\leq  C_t  { \left( \frac{K^{1-\alpha}}{1-\alpha} \right)^2 r_t(N, \delta) }.
\end{equation*}
\end{description}
\end{proposition}

\begin{proof}
\sloppy Let $1<\alpha<2.$ Then $R_t^{N,Y}=B^N_t+M^N_t,$
with $M^N_t \coloneqq M_t(\mu^{N,X},\mu^{N,Y})$
and 
$B^N_t=B_t(X^{N,i},Y^{N,i}, \mu^{N,X},\mu^{N,Y}).$ 
Using first \eqref{eq:b2}  of Lemma \ref{lem:controlb}  and then Proposition \ref{prop:secondrepr}, we obtain
\begin{equation}\label{eq:bb}
\E[|B^N_t|^{\ap}\indiq_{\{t < T^N_K\}}]\leq C\int_0^t\E[|X^{N,i}_s-Y^{N,i}_s|\indiq_{\{s < T^N_K\}}]ds\leq C t r_t^{X,Y}.
\end{equation}

Moreover, using Lemma \ref{lem:controlM}, Jensen's inequality, the exchangeability of $(X^{N,i},Y^{N,i})_i$ and finally once more the boundedness and the Lipschitz continuity of $f,$
\begin{equation}\label{eq:M}
\begin{split}
\E\left[|M^N_t|^{\ap}\indiq_{\{t < T^N_K\}}\right] & \leq C \frac{K^{\ap-\alpha}}{\ap-\alpha} \int_0^t\E\left[\left|\mu^{N,X}_s(f)-\mu^{N,Y}_s(f)\right|^{\ap} \indiq_{\{s < T^N_K\}}\right]ds \\
%& \le C \frac{K^{\ap-\alpha}}{\ap-\alpha} \int_0^t\E\left[\frac 1N\sum_{j=1}^N\left| f(X_s^{N,i})-f(Y_s^{N,i})\right|^{\ap} \indiq_{\{s < T^N_K\}}\right] \\
& \leq C \frac{K^{\ap-\alpha}}{\ap-\alpha} \int_0^t\E[|f(X^{N,i}_s)-f(Y^{N,i}_{s})|^{\ap} \indiq_{\{s < T^N_K\}}]ds\\
& \leq C \frac{K^{\ap-\alpha}}{\ap-\alpha}  \int_0^t\E[|X^{N,i}_s-Y^{N,i}_{s}|\indiq_{\{s < T^N_K\}}]ds
 \leq C_t  \frac{K^{\ap-\alpha}}{\ap-\alpha}  r_t^{X,Y}.
\end{split}
\end{equation}
Finally, collecting \eqref{eq:bb} and \eqref{eq:M} we obtain the claim  $i).$ 

Let now $0<\am<\alpha<1$. Clearly $R_t^{N,Y}=B^N_t+ \Psi^N_t+M^N_t,$
with $M^N_t \coloneqq M_t(\mu^{N,X},\mu^{N,Y}),$
$\Psi^N_t \coloneqq \Psi_t(X^{N,i},Y^{N,i}, \mu^{N,X},\mu^{N,Y})
$ 
and
$B^N_t \coloneqq B_t(X^{N,i},Y^{N,i}, \mu^{N,X},\mu^{N,Y}).$
Now we have, by \eqref{eq:b3} of Lemma \ref{lem:controlb} and \eqref{eq:controlpsi2} of  Lemma \ref{lem:controlpsi} respectively, 
\[
\E[|B_t^N|\indiq_{\{t < T^N_K\}}] \leq C \int_0^t\E[d_{\am}(X_s^{N,i}, Y_s^{N,i}) \indiq_{\{s < T^N_K\}}] ds 
\]
and
\[
\E[|\Psi^N_t|\indiq_{\{ t < T^N_K\}}]\leq C\int_0^t\E[ d_{\am}(X_s^{N,i},Y_s^{N,i}) \indiq_{\{s < T^N_K\}}]ds , 
\]
so that, using Proposition \ref{prop:secondrepr} in both the last inequalities, we obtain
\[ \E[(|B_t^N|+ |\Psi^N_t|)\indiq_{\{t < T^N_K\}}]\leq C_t  \frac{K^{1-\alpha}}{1-\alpha} r_t(N, \delta) .  \]
Finally, \eqref{eq:controlI2} implies
\begin{multline*}
\E[| M^N_t|\indiq_{\{t < T^N_K\}}]\leq  C \frac{K^{1-\alpha}}{ 1- \alpha}  \int_0^t\E\left[W_{d_\am}(\mu_s,\tilde \mu_s)\indiq_{\{s < T^N_K\}}\right]ds
\leq C_t \left(\frac{K^{1-\alpha}}{1-\alpha} \right)^2 r_t(N, \delta) ,
\end{multline*}
concluding the proof.
\end{proof}

\subsection{Representation for the limit system}\label{subs:aX}
In this subsection we prove the following representation result. 
\begin{proposition}\label{prop:MFbarX}
Grant Assumptions \ref{ass:b}--\ref{ass:nu0+}. 
Let $\bar X^N$ denote the unique solution to the limit system \eqref{eq:limitsystem} driven by $S^{N,\alpha}$, and let $(\bar X^{N,i})_{i=1,\dots,N}$ be the first $N$ coordinates of this solution. Write $ \mu^{N, \bar X^N} $ for the empirical measure of these first $N$ coordinates. Then 

\begin{multline*}
\bar X^{N,i}_t =  X^{i}_0 +   \int_0^t  b(\bar X^{N,i}_s ,  \mu^{N,\bar X^N}_s )  ds + \int_{[0,t]\times\R_+ } \psi (\bar X^{N, i}_{s-}, \mu^{N,\bar X^N}_{s-}(f) )  \indiq_{ \{ z \le  f ( \bar X^{N, i}_{s-}) \}} \bar\pi^i (ds,dz)\\
+  
 \int_{[0,t] } \left( \mu^{N,\bar X^N}_{s-}(f)  \right)^{1/ \alpha}   d S_s^{ N,\alpha}+R_t^{N,\bar X}, 
\end{multline*} 
where
\begin{description}
\item{i)} if  $1<\alpha<\ap<2,$ 
\begin{equation*}\label{eq:rbarx}
\E\left[|R_t^{N,\bar X}|^{\ap}\indiq_{\{t < T^N_K\}}\right]\leq  C_t 
\frac{ K^{\ap-\alpha}}{\ap - \alpha} N^{-1/2};
\end{equation*}
\item{ii)}  if $0<\am<\alpha<1,$
\begin{equation*}\label{eq:rbarx}
\E\left[|R_t^{N,\bar X}|\indiq_{\{t < T^N_K\}}\right]\leq  C_t
\frac{ K^{1-\alpha}}{1 - \alpha} (  N^{ - 1/2} \indiq_{ \{1 > \am > \frac12\}} +
N^{- \am }\indiq_{\{ \am < \frac12\}}) .
\end{equation*}
\end{description}

\end{proposition}

To prove the above result we will need the following lemma.
\begin{lemma}\label{lem:fournier_guillin}
Supposing that $ \E ( |X^i_0|^{ (2 \alpha ) \vee 1} ) < \infty $ if $ \alpha < 1 $ and $ \E ( | X^i_0|^p ) < \infty $ for some $p > 2 $ in case $ \alpha > 1, $ 
we have for any $t\geq 0$ and $ \am < \alpha $ 
\begin{equation*}
\E ( W_\am ( \mu^{N,\bar X^N}_t, \bar \mu_t ) ) \indiq_{ \alpha < 1 } + \E ( W_1(\mu^{N,\bar X^N}_t, \bar \mu_t)) \indiq_{ \alpha > 1 }  \le C_t 
 \left\{ 
\begin{array}{ll}
 N^{ - 1/2} & \alpha > 1 \\
 N^{ - 1/2}, & 1 > \am > \frac12 \\
 N^{- \am }, & \am < \frac12 
\end{array} 
\right\} ,
\end{equation*}
where $C_{ t } $ is a positive constant which is non-decreasing in $t .$ 
\end{lemma}
The proof of this lemma is given in  the Appendix section \ref{app:proof_fg}. 

\begin{proof}[Proof of Prop. \ref{prop:MFbarX}]
Here, with $\Psi=0$ in the case $1<\alpha<2,$
 \[ R_t^{N,\bar X}=B_t(\bar X^{N,i},\bar X^{N,i}, \mu^{N,\bar X^N},\bar \mu)+\Psi_t(\bar X^{N,i},\bar X^{N,i}, \mu^{N,\bar X^N},\bar \mu)+M_t( \mu^{N,\bar X^N},\bar \mu).\]
We start with the proof of the case $1<\alpha<2.$ 
Using Lemma \ref{lem:controlb},
\begin{equation*}
\E[|B_t(\bar X^{N,i},\bar X^{N,i}, \mu^{N,\bar X^N},\bar \mu)|^{\ap}\indiq_{\{s < T^N_K\}}]\leq 
\int_0^t\E\left[W_1(\mu^{N,\bar X^N}_s ,\bar \mu_s)\right]ds\leq CtN^{-1/2}.
\end{equation*}

Using Lemma \ref{lem:controlM} and Lemma \ref{lem:fournier_guillin}, the result of $i)$ follows from
\begin{equation*}\label{eq:Mgron}
\E\left[|M_t( \mu^{N,\bar X^N},\bar \mu)|^{\ap}\indiq_{\{t < T^N_K\}}\right] 
\leq C \frac{K^{\ap-\alpha}}{\ap-\alpha}\int_0^t\E\left[W_1(\mu^{N,\bar X^N}_s,\bar \mu_s)\right]ds\leq  C_t \frac{K^{\ap-\alpha}}{\ap-\alpha} N^{-1/2}.
\end{equation*}
We now deal with the case $0<\alpha<1.$ From Lemma \ref{lem:controlb} and Lemma \ref{lem:fournier_guillin}, 
\begin{multline*}
\E[|B_t(\bar X^{N,i},\bar X^{N,i}, \mu^{N,\bar X^N},\bar \mu)|\indiq_{\{s < T^N_K\}}]\\
\leq  C
\int_0^t\E\left[W_{d_\am}(\mu^{N,\bar X^N}_s,\bar \mu_s)\right]ds \le  C
\int_0^t\E\left[W_{\am}(\mu^{N,\bar X^N}_s,\bar \mu_s)\right]ds
\\
\leq C t(  N^{ - 1/2} \indiq_{ \{1 > \am > \frac12\}} +
N^{- \am }\indiq_{\{ \am < \frac12\}}) .
\end{multline*}
The same bound is true for $\Psi_t(\bar X^{N,i},\bar X^{N,i}, \mu^{N,\bar X^N},\bar \mu),$ by Lemma \ref{lem:controlpsi}.
Moreover, from Lemma \ref{lem:controlM} and Lemma \ref{lem:fournier_guillin}, 
\begin{multline*}
\E\left[|M_t( \mu^{N,\bar X^N},\bar \mu)\indiq_{\{t < T^N_K\}}\right]\leq  C \frac{K^{1-\alpha}}{1-\alpha} \int_0^t\E[W_{d_\am}(\mu_s^{N,\bar X^N},\bar \mu_s)]ds \\
\leq \frac{K^{1-\alpha}}{1-\alpha} C t(  N^{ - 1/2} \indiq_{ \{1 > \am > \frac12\}} +
N^{- \am }\indiq_{\{ \am < \frac12\}}),
\end{multline*}
concluding the proof.

\end{proof}

\subsection{Bounding the distance between $Y^{N,i}_t$ and $\bar X^{N,i}_t$}\label{subs:Y_X_lim}
\begin{proposition}\label{prop:dist_Y_barX}
Grant Assumptions \ref{ass:b}--\ref{ass:psi},

Then, for all $N\in \N^\ast$ and $i=1,\dots,N$, 
\begin{description}
\item{i)} For $1<\alpha<2$, for all $\alpha < \ap < 2$ and for all $K > 0$,
\begin{equation*}
\E\left[|Y_t^{N,i}-\bar X_t^{N,i}|^{\ap}\indiq_{\{t < T^N_K\}}\right]\leq   (r_t^{X,Y}+
N^{-1/2})e^{Ct \frac{K^{\ap-\alpha}}{\ap-\alpha} } . 
\end{equation*}
\item{ii)} For $0<\alpha<1$, for all $0<\am<\alpha$ and all $K > 0$, 
\begin{equation*}
\E\left[|Y_t^{N,i}-\bar X_t^{N,i}|\indiq_{\{t < T^N_K\}}\right]\leq
 (r_t(N, \delta) +N^{ - 1/2} \indiq_{ \{\alpha > \am > \frac12\}} +
N^{- \am }\indiq_{\{ \am < \frac12\}})e^{C t \frac{K^{1-\alpha}}{1 - \alpha}}.
\end{equation*}
\end{description}
\end{proposition}

\begin{proof}
Using the representations of $\bar X$ and of $Y$ given by Propositions \ref{prop:MFbarX} and \ref{prop:YtoY},
\begin{multline*}
Y_t^{N,i}-\bar X_t^{N,i}=B(Y_t^{N,i},\bar X_t^{N,i},\mu^{N,Y},\mu^{N,\bar X^N})+\Psi(Y_t^{N,i},\bar X_t^{N,i},\mu^{N,Y},\mu^{N,\bar X^N})\\
+M_t(\mu^{N,Y},\mu^{N,\bar X^N})+R_t^{N,\bar X}+R_t^{N,Y}.
\end{multline*}
We start with the proof of the case $1<\alpha<2.$ In this case
 \[\E\left[|R_t^{N,\bar X}|^{\ap}\indiq_{\{t < T^N_K\}}\right]\leq C_t \frac{ K^{\ap-\alpha}}{\ap - \alpha} N^{-1/2 } \]
 and
 \[ \E[|R_t^{N,Y}|^{\ap}\indiq_{\{t < T^N_K\}}]\leq C_t \frac{K^{\ap-\alpha}}{\ap-\alpha} r_t^{X,Y}  .\]
Using Lemma \ref{lem:controlb}, 
\begin{equation*}
 \E\left[|B(Y_t^{N,i},\bar X_t^{N,i},\mu^{N,Y},\mu^{N,\bar X^N})|^{\ap}\indiq_{\{t < T^N_K\}}\right] 
\leq C\int_0^t\E[|Y_s^{N,i}-\bar X_s^{N,i}|^{\ap}\indiq_{\{s < T^N_K\}}]ds .
\end{equation*}

Furthermore, using the first inequality in \eqref{eq:controlI1} in Lemma \ref{lem:controlM}, Jensen's inequality and exchangeability, 
 \[
 \E\left[|M_t(\mu^{N,Y},\mu^{N,\bar X^N})|^{\ap}\indiq_{\{t < T^N_K\}}\right]\leq C \frac{K^{\ap-\alpha}}{\ap-\alpha}\int_0^t\E\left[|Y_s^{N,i}-\bar X_s^{N,i}|^{\ap}\indiq_{\{s < T^N_K\}}\right] ds.
 \]
We conclude, using Gronwall's lemma, that
\begin{multline*}
\E\left[|Y_t^{N,i}-\bar X_t^{N,i}|^{\ap}\indiq_{\{t < T^N_K\}}\right]\leq C_t  \frac{K^{\ap-\alpha}}{\ap-\alpha} (r_t^{X,Y}+
N^{-1/2})e^{ \frac{K^{\ap-\alpha}}{\ap-\alpha} t}  \\
 \le C  (r_t^{X,Y}+
N^{-1/2})e^{Ct \frac{K^{\ap-\alpha}}{\ap-\alpha} } ,
\end{multline*}
where in the last inequality we have used that $ x e^x \le e^{C x } $ for some $ C> 0, $ for all $ x \geq 0.$

We now deal with the case $0<\alpha<1.$ Then
\[\E\left[|R_t^{N,\bar X}|\indiq_{\{t < T^N_K\}}\right]\leq C_t \frac{ K^{1-\alpha}}{1 - \alpha} (  N^{ - 1/2} \indiq_{ \{\alpha > \am > \frac12\}} +
N^{- \am}\indiq_{\{ \am < \frac12\}})
\]
and
\[\E[|R_t^{N,Y}|\indiq_{\{t < T^N_K\}}]\leq  C_t   \left( \frac{K^{1-\alpha}}{1-\alpha} \right)^2 r_t(N, \delta) .\]
By Lemma \ref{lem:controlb}, 
\[
 \E\left[|B(Y_t^{N,i},\bar X_t^{N,i},\mu^{N,Y},\mu^{N,\bar X^N})|\indiq_{\{t < T^N_K\}}\right]\leq C\int_0^t\E[|Y_s^{N,i}-\bar X_s^{N,i}|\indiq_{\{s < T^N_K\}}]ds . 
\]
The same bound is true for $\Psi$, from Lemma \ref{lem:controlpsi}. Using Lemma \ref{lem:controlM} and the boundedness of $f$, 
 \[
 \E\left[|M_t(\mu^{N,Y},\mu^{N,\bar X^N})|\indiq_{\{t < T^N_K\}}\right]\leq C\frac{ K^{1-\alpha}}{1 - \alpha} \int_0^t\E\left[|Y_s^{N,i}-\bar X_s^{N,i}|\indiq_{\{s < T^N_K\}}\right] .
 \]
We conclude similarly as above, using Gronwall's lemma, that in the case $0<\alpha<1,$
\begin{multline*}
\E\left[|Y_t^{N,i}-\bar X_t^{N,i}|\indiq_{\{t < T^N_K\}}\right]\leq  C_t (r_t(N, \delta) +N^{ - 1/2} \indiq_{ \{\alpha > \am > \frac12\}} +
N^{- \am }\indiq_{\{ \am < \frac12\}})e^{C t \frac{K^{1-\alpha}}{1 - \alpha}}.
\end{multline*}
\end{proof}

\subsection{Proof of Theorem \ref{thm:prop:chaos}}\label{subs:proof:prop:chaos:conclude}
\begin{proof}
We start proving item 2. In case $ \alpha > 1 , $ we have, 
% using Proposition \ref{prop:secondrepr} 
 using Propositions \ref{prop:secondrepr} and \ref{prop:dist_Y_barX} and supposing $ r_t^{X, Y } \le 1 $ (which is true for $ N$ sufficiently large), 
\begin{multline*}
\E[|X^{N,i}_t - \bar X^{N,i}_t|\indiq_{\{t<T^N_K\}}]  
\leq 
  \E[|X^{N,i}_t - Y^{N,i}_t| \indiq_{\{t < T^N_K\}}] +  \left (\E[|Y^{N,i}_t - \bar X^{N,i}_t|^{\ap}\indiq_{\{t<T^N_K\}}]\right)^{1/\ap}\\ \leq C_t r_t^{X,Y}+ 
\left(  (r_t^{X,Y}+
N^{-1/2})e^{Ct \frac{K^{\ap-\alpha}}{\ap-\alpha} }\right)^{1/\ap} \\
\le (r_t^{X,Y}+ N^{-1/2})^{1/\ap}e^{Ct \frac{K^{\ap-\alpha}}{\ap-\alpha} } . 
\end{multline*} 
Let now $0<\alpha<1.$ In this case, using Propositions \ref{prop:secondrepr} and \ref{prop:dist_Y_barX},
\begin{multline*}
\E[d_{\am}(X^{N,i}_t, \bar X^{N,i}_t) \indiq_{\{t<T^N_K\}}]  
\leq 
  \E[d_{\am}(X^{N,i}_t, Y^{N,i}_t)\indiq_{\{t<T^N_K\}} ] +  \E[|Y^{N,i}_t - \bar X^{N,i}_t|\indiq_{\{t<T^N_K\}}]\\
  \leq C_t  \frac{K^{1-\alpha}}{1-\alpha} r_t(N, \delta)  + (r_t(N, \delta)  +N^{ - 1/2} \indiq_{ \{1 > \am > \frac12\}} +
N^{- \am }\indiq_{\{ \am < \frac12\}})e^{C t \frac{K^{1-\alpha}}{1 - \alpha} }\\
\le C \left(  r_t(N, \delta)  +N^{ - 1/2} \indiq_{ \{1 > \am > \frac12\}} +
N^{- \am }\indiq_{\{ \am < \frac12\}}\right) e^{C t \frac{K^{1-\alpha}}{1 - \alpha} } ,
\end{multline*} 
which finishes the proof of item 2. 

We now turn to the proof of item 1. and start discussing the case $ \alpha > 1.$ 
We write
\begin{equation}
  \E ( | X_t^{N, i } - \bar X_t^{N,i} |) =  \E ( | X_t^{N, i } - \bar X_t^{N,i} | \indiq_{\{t \geq T^N_K\}}) +  \E ( | X_t^{N, i } - \bar X_t^{N,i} |\indiq_{\{t<T^N_K\}} ) .
\end{equation}  
We have already achieved a control of the second term of the right hand side of the above inequality (item 2.). Concerning the first term, comparing the equations defining $ X^{N, i }_t $ and $ \bar X^{N i }_t, $ recalling the 
notation of equations \eqref{eq:atn} and \eqref{eq:def:E^Ni_t} and using similar arguments as those leading to \eqref{eq:mufin}, we have
$$ | \bar X_t^{N,i} - X_t^{N, i } |  \le  2  \| b\|_\infty t +  \ \left| \int_{[0, t ] } \bar \mu_{v-} (f)^{ 1/ \alpha} d S_v^{N, \alpha } \right| + |A_t^N | +  |  E_t^{N, i } | ,$$
where both $  \left| \int_{[0, t ] } \bar \mu_{v-} (f)^{ 1/ \alpha} d S_v^{N, \alpha } \right|  $ and $ E_t^{N, i } $  possess a finite moment of order $p$  for any $ 0 < p < \alpha ,$ uniformly in $N.$ 
Moreover, by Theorem \ref{prop:representation_finite_syst}, 
$$ | A_t^N| \le \left| \int_{[0, t]} \left( \mu^N_{\tau_s} (f) \right)^{1/\alpha } d S_s^{N, \alpha } \right| + | R_t^N| =: | J_t^N | + |R_t^N|, $$
where $J_t^N $ possesses again a finite moment of order $p$  for any $ 0 < p < \alpha ,$ uniformly in $N.$ 

So we may write
\begin{multline*}
 | X_t^{N, i } - \bar X_t^{N,i} | \indiq_{\{ T_K^N \le t  \} } \\
 \le \left( 2  \|b\|_\infty t +  \left| \int_{[0, t ] } \bar \mu_{v-} (f)^{ 1/ \alpha} d S_v^{N, \alpha } \right|+ | J_t^N | + |E_t^{N, i } |  \right) \indiq_{\{ T_K^N \le t  \} } +  | R_t^N| \\
 =: {\mathcal T}_t^{N, i } \indiq_{\{ T_K^N \le t  \} } +  | R_t^N| , 
\end{multline*} 
where $ {\mathcal T}_t^{N, i } ,$ the sum of the four terms appearing above,  has finite moments of order $p$ for any $ 0 < p < \alpha, $ uniformly in $N.$ 
Moreover, we know that $ \E |R_t^N | \le C_t r_t (N, \delta), $ thanks to Theorem \ref{prop:representation_finite_syst}. 
Fix now some $ 1 < p < \alpha $ and use H\"older's inequality with conjugate coefficients $p$ and $q$ to obtain that
$$
 \E  \left[  {\mathcal T}_t^{N, i } \indiq_{\{ T_K^N \le t  \} }\right]   
 \le C_{p } \| {\mathcal T}_t^{N, i } \|_p      \mathbf{P} ( T_K^N \le t )^{ 1/q} =: C_p ( t) \mathbf{P} ( T_K^N \le t )^{ 1/q},
$$
where  $  C_p ( t)  $ is the uniform in $N$ bound on the $L_p-$norm of ${\mathcal T}_t^{N, i }.$ Since $ T_K^N $ is exponentially distributed with parameter $ \nu^\alpha ( B_K^c )$ (recall \eqref{eq:St}) and since $ \nu^\alpha ( B_K^c ) \le C K^{ - \alpha} , $ we
have that  
$$ \mathbf{P} ( T_K^N \le t ) = 1 - e^{ - C t K^{- \alpha } }  \to 0 $$ as $ K \to \infty.$ To summarize, we have
$$
 \E ( | X_t^{N, i } - \bar X_t^{N,i} |) \le C_p ( t) \mathbf{P} ( T_K^N \le t )^{ 1/q} + C_t r_t ( N, \delta) 
 + \E ( | X_t^{N, i } - \bar X_t^{N,i} |\indiq_{\{t<T^N_K\}} ) .
$$ 
Now, to prove that $ \E ( | X_t^{N, i } - \bar X_t^{N,i} |)  $ converges to $0, $ as $ N \to \infty, $ we proceed as follows. Fix any $ \varepsilon > 0.$ 
Then we may choose $K$ such that  
$$C_p ( t) \mathbf{P} ( T_K^N \le t )^{ 1/q} \le \varepsilon / 3.$$  
Once $K$ is fixed with this control, we then choose $ N_0 = N_0 (K ) $ sufficiently big such that the right hand side of \eqref{eq:finrate:alpha>1} is bounded by $ \varepsilon / 3 $ for all $ N \geq N_0 (K)$ and such that moreover $ C_t r_t (N, \delta (N)), $ our bound on $\E |R_t^N | , $ is also bounded by $ \varepsilon/3 $ for all $ N \geq N_0.$  Therefore,  $\E ( | X_t^{N, i } - \bar X_t^{N,i} |)  \le \varepsilon $ for all $ N \geq N_0 .$ Since $ W_1  ( {\mathcal L}( X_t^{N, i } ) , {\mathcal L} ( \bar X_t^i ) )  \le \E ( | X_t^{N, i } - \bar X_t^{N,i} |) , $ this implies item 1. in case $ \alpha > 1.$  

We now discuss the case $\alpha <1.$ Since $ \psi $ and $f$ are bounded, using the same arguments and the same notation as in the case $ \alpha > 1,$ we now have 
$$ 
|X^{N, i}_t - \bar X^{N,i}_t  |  \le   2 \| b \|_\infty t +  2  \| \psi\|_\infty N_t^{ \| f\|_\infty } + | A_t^N| + 
| E^{N, i}_t |   +  \left|  \int_{[0,t] } \left( \bar \mu_{s-}(f)  \right)^{1/ \alpha}   d S_s^{N,\alpha} \right| ,$$
where $ N_t^{\|f\|_\infty} = \int_{[0,t]\times\R_+ }   \indiq_{ \{ z \le \|f\|_\infty  \}} \bar \pi^i (ds,dz)  $ has Poisson distribution with parameter $ \|f\|_\infty t .$ We use once more the representation  $A_t^N=J_t^N+R_t^N,   $ obtained in Theorem \ref{prop:representation_finite_syst}.   

Put now 
$${\mathcal T}_t^{N, i} := 2 \| b \|_\infty t +  2  \| \psi\|_\infty N_t^{ \| f\|_\infty }  + |J_t^N| +|E^{i,N}_t| + \left|  \int_{[0,t] } \left( \bar \mu_{s-}(f)  \right)^{1/ \alpha}   d S_s^{N,\alpha}\right| ,$$
which has finite moments of order $ \tilde p, $ uniformly in $ N, $ for any $ 0 < \tilde p < \alpha.$ 
The sub-additivity of  $\R_+ \ni x \mapsto  \| x \|_{d_{\alpha_{-}}}   $ implies then that 
$$
\|X_t^{N,i}-\bar X_t^{N,i}\|_{d_{\alpha_{-}}}  \indiq_{\{ T_K^N \le t  \} } \le \|{\mathcal T}_t^{N, i } \|_{d_{\alpha_{-}}} \indiq_{\{ T_K^N \le t  \} }
  + \|R_t^N\|_{d_{\alpha_{-}}} .
$$
Finally, we choose $1 <  p<\frac {\alpha}{\alpha_{-}},$ use H\"older's inequality with conjugate coefficients $p$ and $q$ and moreover the fact that 
$$(|x|^{\alpha_{-}}\wedge|x|)^p=|x|^{p\alpha_{-}}\wedge|x|^p\leq |x|^{p\alpha_{-}}\wedge|x|, \mbox{ i.e. } (\|x\|_{d_{\alpha_{-}}} )^p\leq  \| x\|_{d_{p \alpha_{-}}} $$
to deduce that 
$$\E(\|{\mathcal T}_t^{N, i } \|_{d_{\alpha_{-}}} \indiq_{\{ T_K^N \le t  \} })\leq[\E (\|{\mathcal T}_t^{N, i } \|_{p d_{\alpha_{-}}})))]^{1/p} \mathbf{P} ({\{ T_K^N \le t  \} })^{1/q}\leq C_{ p \alpha_{-} } (t)  \mathbf{P}({\{ T_K^N \le t  \} })^{1/q},$$ 
where the constant $C_{ p \alpha_{-} } (t) $ does not depend on $N$ and is a bound on the $d_{ p \alpha_{-}}-$moment of $ {\mathcal T}_t^{N, i }, $ which is 
finite because  $p\alpha_{-}<\alpha .$ The conclusion of the proof follows then as in the case $ \alpha > 1.$ 
%$$
% | X_t^{N, i } - \bar X_t^{N,i} | \indiq_{\{ T_K^N \le t  \} } \le \left( 2 [ |X^i_0| + \|b\|_\infty t ] + | J_t^N | + |E_t^{N, i } |  \right) \indiq_{\{ T_K^N \le t  \} } +  | R_t^N| .$$ 
%$$d_{\alpha_{-}(|X_t^{N,i}-\bar X_t^{N,i}|)\leq $$

% since for every $\alpha'<\alpha,$  $\E[|X_t^{N,i}|^{\alpha'}]<\infty,$ $\E[|\bar X_t^{N,i}|^{\alpha'}]<\infty,$ and 
%$$(|x|^{\alpha_{-}}\wedge|x|)^p=|x|^{p\alpha_{-}}\wedge|x|^p\leq |x|^{p\alpha_{-}}\wedge|x|,$$ then $$ \E[(d_{\alpha_{-}}(X_t^{N,i},\bar X_t^{N,i})^p]\leq \E[d_{p\alpha_{-}}(X_t^{N,i},\bar X_t^{N,i})]<\infty.$$ Therefore we can apply Holder inequality to get 
% $$\E ( d_{\alpha_{-}}( X_t^{N, i }, \bar X_t^{N,i} ) ; T_K^N \le t ) \le C  \E(d_{p\alpha_{-}}( X_t^{N, i } ,\bar X_t^{N,i} ))^{1/p}  P( T_K^N \le t )^{ 1/q}, $$
% where again  $ P( T_K^N \le t )= 1 - e^{ - C t K^{- \alpha } }  \to 0 $ as $ K \to \infty$.
%
%and in the case $\alpha <1$ $\E ( d_{\alpha_{-}}( X_t^{N, i }, \bar X_t^{N,i} ) ) \to 0$ as $ K \to \infty $ and $ N = N(K) \to \infty $ slowly enough.

\end{proof}

\subsection{Discussion of the convergence rate obtained in Theorem \ref{thm:prop:chaos}}\label{sec:discussionrate}
We close this section with a discussion of the different ingredients that constitute our rates in \eqref{eq:finrate:alpha>1} and \eqref{eq:finrate:alpha<1}. Taking formally $ \alpha_{-} = \alpha ,$ it is clear that the main contribution to the error arises from $ r_t(N, \delta)$. The term that involves $ g (N \delta ) $ comes from the quantitative version of the stable central limit theorem and can presumably not be improved. The terms $ ( N \delta)^{- 1/2} $ (if $\alpha >1$) and $ ( N \delta)^{- \alpha_{-}/2} $ (if $ \alpha < 1 $) come from the fluctuations in the law of large numbers for the Poisson number of jumps per time interval. So they cannot be improved neither. The $ \delta-$power $ \delta^{1/\alpha } $ (or $\delta^{ \alpha_{-}/\alpha} $) comes from the fact that we discretize time and that we have to deal with the ``overlap'' coming from the stochastic integral against the stable process (that is, the integral between $ K \delta $ and $t , $ where $ K$ is such that $K \delta \le t < (K+1) \delta$). 

Finally, the very first term in the definition of $ r_t(N, \delta)$ comes from the necessity of approximating the interaction term $A_t^N$ of the finite particle system by its time-discretized version. In the $ \alpha_{-} \to \alpha$ limit, it is basically of order $ \delta^{ - 1/\alpha} $ if $ \alpha > 1 $ or $ \delta, $ if $ \alpha < 1, $ which is reasonable. 

To summarize, we believe that the term $ r_t(N, \delta) $ is rather optimal due to the above reasons. 

There are other terms that come into play, like for instance the error that is due to the Wasserstein distance between the empirical measure of the limit system and its theoretical counterpart -- we believe that this cannot be improved neither.

However, we see that in the case $ \alpha > 1, $ it is $ r_t(N, \delta)^{1/\alpha} $ (in the $ \alpha_+ \to \alpha-$ limit) that appears as a final rate. This is because we want to compare the finite system, which is of bounded variation, to the limit one, which is not. Roughly speaking, there are errors that have to be considered in the $L^1 - $norm (such as the one that we control in Proposition \ref{lem:preliminary_results}, in the case $ \alpha > 1$), and other errors that cannot be controlled in $L^1, $ but need to be controlled in another norm, for example $L^{ \alpha_+}, $ for $ \alpha_+ > \alpha > 1$ -- a control that we need to use e.g. in Lemma \ref{lem:controlM} to deal with the stochastic integral terms, when integrating against the stable process. To switch from one norm to another, using H\"older's inequality, we loose with respect to the accuracy of our error. And so it seems rather clear that the final rate of convergence has no reason to be optimal in the case $ \alpha > 1.$

\section{Appendix}
\subsection{Proof of Theorem \ref{thm:existfinite}, strong existence and uniqueness for the finite system }\label{subsec:strongexistfin} 

\begin{proof} 
Fix $N \in \N^\ast$ and let $\tau_0\coloneqq 0$ and $(\tau_n)_{n\geq 1}$ be the sequence of the jump times of the Poisson process $ (\sum_{j=1}^N \pi^j ( [0, t ] \times [0, \|f\|_\infty ] \times \R ))_{t \geq 0} .$ Let moreover $(U^k)_{k\geq 1}$ be an i.i.d. family of real-valued random variables, $\sim \nu,$ such that $ U^k $ is the atom of $ \sum_{j=1}^N \pi^j  ( \{ \tau_k\} \times [0, \|f\|_\infty ], \cdot ) .$ 
We construct the solution to \eqref{eq:micro:dyn} recursively on each interval $[\tau_n, \tau_{n+1}]$, $n\geq 0$. On $[\tau_0,\tau_1)$, the solution to \eqref{eq:micro:dyn} obeys 
\begin{equation}\label{eq:micro.dyn:no:jumps}
    \begin{split}
        X^{N,i}_t = X^{i}_0 + \int_0^t b(X^{N,i}_s, \mu^N_s) ds.
    \end{split}
\end{equation}
By Assumption \ref{ass:b}\ref{itm:bLip}, \eqref{eq:micro.dyn:no:jumps} admits a unique solution on $ [ 0, \tau_1 [ $ (\cite{graham92}). 
Suppose that at $\tau_1$, the $i$-th particle has a main jump. We put then $X^{N,i}_{\tau_1} = X^{N,i}_{\tau_1 - } + \psi (X^{N,i}_{\tau_1 - }, \mu^N_{\tau_1-} ) $ in case $ \alpha < 1 $ and  $X^{N,i}_{\tau_1} = X^{N,i}_{\tau_1 - } $ if $ \alpha > 1 .$ Moreover, we put $X^{N,j}_{\tau_1} = X^{N,j}_{\tau_1- } + \frac{U_1}{N^{1/\alpha}}$ for all $j\neq i$.    
We then solve the equation \eqref{eq:micro.dyn:no:jumps} on $[0, \tau_2-\tau_1]$  with this new initial condition and so on. Since $N$ is finite, $ \tau_n \to \infty $ as $ n \to \infty , $ such that the above construction can be achieved on the whole positive time axis. 

Last, by Assumption \ref{ass:nu0} and since $b$ and $\psi$ are bounded and the $(U_k)_{k\geq 1}$ have finite moments of order $p$ for all $p<\alpha$, we conclude that, for any fixed $N\in\N^\ast$, $t>0$ and $i=1,\ldots,N$, $X^{N,i}_t$ has finite moments of order $p$ for all $p<\alpha$. 
\end{proof}
\subsection{Proof of the strong existence of the limit system}\label{app:strong_ex_lim}
In this subsection we prove strong existence for the limit system \eqref{eq:limitequation1}. We start discussing the case $ \alpha > 1 $ when there are no main jumps. Our argument is partially inspired by the one used in Proposition 2 in \cite{fournier}. Fix a truncation level $K>0,$ recall that we assume $X_0 \in L^2$ and that we have introduced the constant $M_K$ in \eqref{eq:def:C_K}.  We define the following Picard iteration for all $ n \geq 1 :$   
\begin{eqnarray}\label{eq:def:picard} 
  X^{[0],K}_t &\equiv& X_0 \quad \forall \, t\geq 0,\nonumber \\
   X^{[n],K}_0 & =& X_0 \quad \forall\, n \geq 1, \nonumber \\
    X^{[n],K}_t &= &X^{[n-1],K}_0 + \int_0^t b(X^{[n-1],K}_s,\mu^{[n-1],K}_s) ds + \int_{[0,t]\times B^\ast_K} \! \!  \left(\mu^{[n-1],K}_{s -}(f)\right)^{1/\alpha} z \tilde M(ds,dz) \nonumber \\
   &     &  -M_K \int_0^t \left(\mu^{[n-1],K}_s(f) \right)^{1/\alpha} ds   ,
\end{eqnarray}
where $\mu^{[n],K}_s \coloneqq \mathcal{L}(X^{[n],K}_s \, | \, S^\alpha)$. 

Using similar arguments as in the proof of the uniqueness, we have the a priori bound 
\begin{equation}\label{eq:a_priori_bound}
    \begin{split}
        \E\left[\sup_{t \in [0,T]} \left| X^{[n],K}_t \right|^2\right] & \leq  C_{K,T}(1+\E[|X_0|^2])
    \end{split}
\end{equation}
for some constant $C_{K,T}$ which does only depend on the truncation level $K$ and on $T,$ but not on $n$.

We show now that, for any $T>0$ and $t \in [0,T]$, the sequence $(X^{[n],K}_t)_{n\geq 0}$ defined in \eqref{eq:def:picard} converges a.s. to a limit $\bar{X}^K_t$.

First notice that, analogously to what we obtained in the uniqueness proof (see \eqref{eq:uniqueness:gronwall}), 
\begin{equation}\label{eq:previous_bound}
    \begin{split}
        \E\left[\sup_{t \in [0,T]} \left| X^{[n+1],K}_t - X^{[n],K}_t\right|^2\right] & \leq C_{K,T} \int_0^T \E\left[\sup_{t \in [0,s]} \left| X^{[n],K}_t - X^{[n-1],K}_t\right|^2\right] ds
    \end{split}
\end{equation}
for some constant $C_{K,T}$ non-decreasing with respect to $T$.  

Introduce for any $ n \geq 0$ 
\begin{equation}
u^{[n],K}_T \coloneqq \E\left[\sup_{t \in [0,T]} \left| X^{[n+1],K}_t - X^{[n],K}_t\right|^2\right]  
\end{equation}
and iterate \eqref{eq:previous_bound} to obtain
\begin{multline*}
         u^{[n],K}_T  \leq C_{K,T}^2 \int_0^T ds_1 \int_0^{s_1} u^{[n-2],K}_{s_2} ds_2
          \leq C_{K,T}^n \int_0^T ds_1 \int_0^{s_1} ds_2 \dots \int_0^{s_{n-1}} ds_{n} u^{[0],K}_{s_n} \\
          \leq    \frac{T^n}{ n !} C_{K,T}^n u^{[0],K}_{T}, 
\end{multline*}
where $u^{[0],K}_{T} = \E\left[\sup_{t \in [0, T]} \left|X^{[1],K}_t - X^{[0],K}_t\right|^2\right] \le C_T  $ is bounded thanks to \eqref{eq:a_priori_bound}.

This implies that 
\begin{multline*}
\E\left[ \sum_{n\geq 0} \sup_{t \in [0,T]} \left| X^{[n+1],K}_t - X^{[n],K}_t\right| \right]   \leq \sum_{n\geq 0} \sqrt{ \E\left[\sup_{t \in [0,T]} \left| X^{[n+1],K}_t - X^{[n],K}_t \right|^2\right]} \\
 = \sum_{n\geq 0}\sqrt{u^{[n],K}_T} \le \sqrt{C_T}  \sum_{n\geq 0} C_{K,T}^{n/2} \frac{T^{n/2}}{\sqrt{n !}} < +\infty , 
\end{multline*}
such that
\begin{equation*}
\sup_{t \in [0,T]} \sum_{n\geq 0}  \left| X^{[n+1],K}_t - X^{[n],K}_t\right| 
< +\infty \qquad \text{a.s..}
\end{equation*}
Hence the series $\sum_{n\geq 0} \left( X^{[n+1],K}_t - X^{[n],K}_t\right) $ converges a.s. and we can define a.s. 
\begin{equation*}
\bar{X}^K_t \coloneqq X_0 + \sum_{n\geq 0} \left( X^{[n+1],K}_t - X^{[n],K} _t\right) .
\end{equation*}

The next step is to prove that the a.s. limit of $(X^{[n],K}_t)_{n\geq 0}$, $\bar X^K_t$, solves the same equation as the process $\bar X_t$ on $[0,  T_K[,$ that is, that almost surely,
\begin{eqnarray}\label{eq:picard_limit:fixed_K}
\bar X^K_t & =& X_0 + \int_0^t b(\bar X^K_s, \bar \mu^K_s) ds + \int_{[0,t]\times B^\ast_K} \left(\bar \mu^K_{s-}(f)\right)^{1/\alpha} z \tilde M(ds,dz) \nonumber \\
        && \qquad -M_K \int_0^t \left(\bar \mu^K_s(f) \right)^{1/\alpha} ds ,
\end{eqnarray}
where $\bar \mu^K_s \coloneqq \mathcal{L}(\bar X^K_s \, | \, S^\alpha)$. 
This follows by taking the $n \to +\infty$ limit in \eqref{eq:def:picard}, since 
\begin{itemize}
\item by Assumption \ref{ass:b}\ref{itm:bLip}, $\E[|b(X^{[n],K}_t, \mu^{[n],K}_t) - b(\bar X^K_t, \bar \mu^K_t)|] \leq C \E[|X^{[n],K}_t -\bar X^K_t|]$. 
On the other hand, $\E[|X^{[n],K}_t -\bar X^K_t|] \to 0$ by dominated convergence.  
Hence, from this $L^1$ convergence of $X^{[n],K}_t$ to $\bar X^K_t$,  we obtain $L^1$ convergence of  $b(X^{[n],K}_t, \mu^{[n],K}_t)$ to $b(\bar X^K_t, \bar \mu^K_t)$. This latter yields in turn a.s. convergence, up to a subsequence. 
\item $\mu^{[n],K}_t(f)  = \E\left[f(X^{[n],K}_t) \, | \, S^\alpha \right] \to \E\left[f(\bar X^K_t) \, | \, S^\alpha\right] \eqqcolon \bar\mu^K_t(f)$ a.s.  
Moreover, we have that  $| \mu^{[n],K}_t(f) |\leq \norme{f}_\infty$ by Assumption \ref{ass:f}\ref{itm:fbdd}, hence we obtain the $L^2$-convergence of the stochastic integrals against $ \tilde M,$ whence the almost sure convergence, once again for a subsequence. The convergence of $ \int_0^t \left(\mu^{[n],K}_s(f) \right)^{1/\alpha} ds$ follows similarly.
\end{itemize}

As a consequence of the above construction we dispose of a family of processes $(\bar X^{K})_{K \in \N^*}$ such that
\begin{itemize}
\item[(i)] for any $K$, $\bar X^K$ solves \eqref{eq:picard_limit:fixed_K};  
\item[(ii)] $\bar X^{K+1}_t = \bar X^K_t$ a.s. for all $t\in [0,T_K[ ,$ since both are solution of the same equation on $ [0, T_K[.$  
\end{itemize} 

So, letting $T_0 = 0$, the following process is well-defined 
\[
\bar X_t \coloneqq \sum_{K \geq 1} \indiq_{[T_{K-1}, T_{K}[}(t) \bar X^{K}_t, 
\]
and it solves \eqref{eq:limitsystem} on $[0,T]$ for any $T\geq 0$.

The same construction thanks to a Picard iteration works also in the case $ \alpha < 1, $ using $L^1$-norm instead of $L^2$-norm now. Details are omitted.

\subsection{Proof of Lemma \ref{lem:fournier_guillin}}\label{app:proof_fg}
\begin{proof}
\sloppy Under the law $ \mathbf{P} ( \cdot | S^{N, \alpha}), $ that is, conditionally on $S^{N, \alpha}, $ the $N$ coordinates $ \bar X^{N, 1 }_t , \ldots, \bar X^{N, N}_t $ are i.i.d. and distributed according to $ \bar \mu_t. $ 

Let us first treat the case $ \am  < \frac12.$ We have already argued in  Remark \ref{rem:33}, see in particular \eqref{eq:mufin2}, that $ \bar \mu_t $ admits a finite first moment. Theorem 1 of \cite{fournierguillin} (with their $p$ replaced by $ \am $ and their $q$ replaced by $1$, and using conditional expectation $ \E ( \cdot | S^{N, \alpha})$ instead of unconditional one) implies that, for a universal constant $ C( \am ), $   
\[ \E \left(W_ \am  (  \mu_t^{N, \bar X^N}, \bar \mu_t ) | S^{N, \alpha } \right) \le C ( \am )  \left( \int |x| \bar \mu_t ( dx)\right)^ \am  \cdot 
N^{-  \am  } . \] 
Relying on the upper bound obtained in Remark \ref{rem:33} above, we have that 
$  \int |x| \bar \mu_t ( dx) \le \E ( |X_0 | ) + Ct + \|f\|_\infty  \sup_{s \le t } | S_s^{N, \alpha}| .$ 
Using the sub-additivity of the function $ |\cdot |^ \am  $ and taking expectation then yields the result.

We now treat the case $ 1 >  \am  > \frac12, $ in which case $ \alpha > \frac12 $ as well, since $ \alpha >  \am .$ In this case, $ \bar \mu_t $ admits a finite moment of order $ 2 \alpha ,$ since $ X_0 $ does by assumption. We now apply Theorem 1 of \cite{fournierguillin} with their $p$ replaced by $ \am $ and their $q$ replaced by $2 \alpha$ such that  
\[ \E \left(W_ \am  (  \mu_t^{N, \bar X^N}, \bar \mu_t ) | S^{N, \alpha } \right) \le C ( \am , \alpha )  \left( \int |x|^{2 \alpha} \bar \mu_t ( dx)\right)^{ \am  /(2 \alpha)} \cdot 
( N^{- 1/2 } + N^{ - ( 1 - \frac{ \am }{{2 \alpha}} )} )  . \] 
Since $ \frac12 < 1 - \frac{ \am }{{2 \alpha}} , $ the leading order of the above expression is given by $ N^{ - 1/2}.$ Moreover, 
\[  \int |x|^{2 \alpha} \bar \mu_t ( dx) \le C_t ( \E ( |X_0 |^{2 \alpha} ) + 1 + \sup_{s \le t } | S_s^{N, \alpha}|^{2 \alpha}  .\] 
The sub-additivity of the function $ | \cdot |^{  \am  / ( 2 \alpha ) } $ (recall that $  \am  < \alpha $) implies then that, yet for another constant $ \tilde C_t,$  
\[ \left( \int |x|^{2 \alpha} \bar \mu_t ( dx)\right)^{ \am  /(2 \alpha)}  \le \tilde C_t ( 1 + \sup_{s \le t } | S_s^{N, \alpha}|^ \am ) , \]
and we take expectation to conclude. 

Finally in case $ \alpha > 1,$ we use Theorem 1 of \cite{fournierguillin} with their $p$ replaced by $1$ and their $q$ replaced by our $p$ such that   
\[ \E \left(W_1 (  \mu_t^{N, \bar X^N}, \bar \mu_t ) | S^{N, \alpha } \right) \le C (p)  \left( \int |x|^p \bar \mu_t ( dx)\right)^{1/p} \cdot 
N^{- 1/2  } , \]
and we conclude similarly as above, using the sub-additivity of the function $ |\cdot |^{1/p}. $  
\end{proof}

\subsection{Proof of Remark \ref{rem:rate}}\label{sec:delta_choice}

\textbf{Case $\pmb{1<\alpha<2}$. }

According to Theorem \ref{thm:prop:chaos}, the error term is given, up to a constant, by  \eqref{eq:finrate:alpha>1}, where we have to choose $ \delta = \delta (N) $ such that $ N \delta \to \infty.$

To understand formally what is the leading term in our error, we let $ \am \uparrow \alpha $ and $\ap \downarrow \alpha$. In the limit $\am=\ap = \alpha$, we are left with an error term given by (up to a constant and to the common power $1/\alpha$) 
\[ \delta^{\frac{1}{\alpha^2}}
 + \delta^{\frac{1-\alpha}{\alpha}}\left(   g(N \delta)+  (N\delta)^{- 1/2} \right)   .\]

Since $g(x) = x^{-B}$ for some $B>0$, if we suppose $\delta = N^{-\eta}$ for some $\eta \in (0,1)$, we can write 
\[ \delta^{\frac{1-\alpha}{\alpha}}\left(   g(N \delta)+  (N\delta)^{- 1/2} \right) = N^{ \eta ( 1 - 1/ \alpha + B ) } g(N) + N^{ \eta(  3/2 - 1/ \alpha ) } N^{ - 1/2} , \]
which is an increasing function of $\eta.$ 
On the other hand, $\delta^{\frac{1}{\alpha^2}}$  is a decreasing function of $\eta, $ and so we have to choose $\eta$ such that the two terms which are left are equal, that is, 
\[ \delta^{ \frac {1-\alpha + \alpha^2}{\alpha^2}} =    g(N \delta)+  (N\delta)^{- 1/2}  .\] 
The leading term between $g(N\delta)$ and $(N\delta)^{-1/2}$ will asymptotically behave as $(N \delta)^{-C}$, with either $C=B$ or $C= 1/2$. 
Then we have to solve
\[ \delta^{ \frac {1-\alpha + \alpha^2}{\alpha^2}} = N^{-C} \delta^{-C} ,
\mbox{ which gives } \delta = N^{ - \frac{C\alpha^2}{1 -\alpha + C\alpha^2 + \alpha^2 } }  .\]
By the equality we imposed, and re-introducing the $1/\alpha$ power, the rate will be 
\[ \delta^{ 1 /\alpha^3 } =  N^{ - \frac{C}{(1-\alpha+C\alpha^2 + \alpha^2 )\alpha   } } .\]

Consider now the explicit form of the function $g$ given in \eqref{eq:def:g}. We give the explicit rate (i.e. we choose $C$) in the two cases: 

\textbf{Case 1:} $\gamma < 2-\alpha$. Then, if $\gamma < \alpha/2$, $C= \gamma/\alpha$ and the rate is 
\[  N^{ - \frac{\gamma}{ \alpha^2 (1-\alpha +\gamma\alpha + \alpha^2 )} } ,\] 
whereas, if $\gamma \in \left(\frac{\alpha}{2}, 2-\alpha\right)$, $ C = 1/2,$ such that the rate is 
\[  N^{ - \frac{1}{2\alpha \left(1 -\alpha + \frac{3}{2}\alpha^2 \right)} } .\]

\textbf{Case 2:} $\gamma > 2- \alpha$. Then, for $ \alpha \in ( 1, \frac43 ) , $ $ C = 1/2,$ such that the rate is once more 
\[ N^{ - \frac{1}{2\alpha \left(1 -\alpha + \frac{3}{2}\alpha^2 \right)} } .\]
If $\alpha > \frac{4}{3}$, then $C = \frac{2-\alpha}{\alpha}$, such that the rate is
\[   N^{-\frac{2-\alpha}{\alpha^2(1+\alpha)}} . \]

\textbf{Case $\pmb{\alpha<1}$.}

According to Theorem \ref{thm:prop:chaos}, the error term equals, up to a constant, the expression  \eqref{eq:finrate:alpha<1}. Taking $\am=\alpha$ gives the sharpest possible bound, which includes the terms 
\[
\delta \qquad g(N\delta) \qquad (N\delta)^{-\frac{\alpha}{2}} 
\]
where $ (N\delta)^{-\frac{\alpha}{2}}$ dominates the term $(N\delta)^{-1}$ which is present in the expression of $g$ given in \eqref{eq:def:g:alpha<1}.  
Hence we are left with an error proportional to 
\[
\delta  + \left[(N\delta)^{-\frac{\gamma}{\alpha}} + (N\delta)^{\frac{\alpha-1}{\alpha}}\right] 
 + (N\delta)^{-\frac{\alpha}{2}} .
\]

Take now $\delta = N^{-\eta}$ with $0<\eta<1$. Then $ (N\delta)^{-\frac{\gamma}{\alpha}} + (N\delta)^{\frac{\alpha-1}{\alpha}} + (N\delta)^{-\frac{\alpha}{2}}$ is an increasing function of $\eta$ whose leading order has the form $(N\delta)^{-C}$ with $C$ one of those exponents, while $\delta$ is a decreasing function of $\eta.$ Therefore, we impose
\[
\delta = (N\delta)^{-C}, \mbox{ whence } \delta = N^{-\frac{C}{1+C}} .
\]
We conclude by choosing $C$ according to \eqref{eq:def:g:alpha<1}.
If $\alpha \in (0,\sqrt{3}-1)$ and $\gamma < \frac{\alpha^2}{2}$ or if $\alpha \in (\sqrt{3}-1,1)$ and $\gamma < 1-\alpha$, the rate is 
\[ N^{-\frac{\gamma}{\alpha+\gamma}} .\]
If $\alpha \in (0,\sqrt{3}-1)$ and $\gamma > \frac{\alpha^2}{2}$, the rate is 
\[  N^{-\frac{\alpha}{2+\alpha}} .\]
If $\alpha \in (\sqrt{3}-1,1)$ and $\gamma > 1-\alpha$, the rate is
\[ N^{\alpha-1}. \]

\subsection{A coupling lemma}\label{app:coupling}
We recall here for convenience Lemma 3.12 in \cite{prodhommeArxiv}, which is proven there. 
\begin{lemma}\label{lem:prodhommeArxiv}
Let $E_1$ and $E_2$ be two complete separable metric spaces, let $\mu$ be a probability distribution on $(E_1 \times E_2, \mathcal{B}(E_1) \otimes \mathcal{B}(E_2))$. Let $\mu_1$ denote the first marginal of $\mu$. There exists a measurable function $G : E_1 \times (0,1) \to E_2$ such that if $(X_1, V) \sim \mu_1 \otimes U(0,1) $, then  $(X_1, G(X_1, V)) \sim \mu$.  
\end{lemma}

%%%%%%%%%%%%%%%%%%%%%%%%%%%%%%%%%%%%%%%%%%%%%%%%%%%%%%%%%%%%%%%%%%%
%%                                                               %%
%% Use the two commands below for producing your bibliography    %%
%% with bibtex, then comment again the commands and include the  %%
%% content of the .bbl file in this file below the commands.     %%
%%                                                               %%
%%%%%%%%%%%%%%%%%%%%%%%%%%%%%%%%%%%%%%%%%%%%%%%%%%%%%%%%%%%%%%%%%%%

% \bibliographystyle{amsplain}
% \bibliography{biblio}

% add below the content of your .bbl file produced by bibtex.

%%%%%%%%%%%%%%%%%%%%%%%%%%%%%%%%%%%%%%%%%%%%%%%%%%%%%%%%%%%%%%%%%%%
%%                                                               %%
%% You may add acknowledgments (optional).                       %%
%%                                                               %%
%%%%%%%%%%%%%%%%%%%%%%%%%%%%%%%%%%%%%%%%%%%%%%%%%%%%%%%%%%%%%%%%%%%
\begin{acks}
E.L. and D.L. acknowledge support of the Institut Henri Poincar\'e (UAR 839 CNRS-Sorbonne Universit\'e), and LabEx CARMIN (ANR-10-LABX-59-01). This work has been conducted as part of  the ANR project ANR-19-CE40-0024. E.M. acknowledges financial support from Progetto Dottorati - Fondazione Cassa di Risparmio di Padova e Rovigo and from the ANR grant ANR-21-CE40-0006 SINEQ. We thank the two anonymous referees and Jan Swart for their careful reading and the many useful remarks that have helped us to improve the manuscript. 
\end{acks}

\end{document}